\documentclass[
  a4paper, 
  reqno, 
  oneside, 
  12pt
]{amsart}

\usepackage[utf8]{inputenc}

\usepackage[dvipsnames]{xcolor} % Must come before tikz!
\colorlet{cite}{red}
\usepackage{tikz}  % for diagrams
\usetikzlibrary{arrows,positioning}
\tikzset{ 
  baseline=-2.3pt,
  text height=1.5ex, text depth=0.25ex,
  >=stealth,
  node distance=2cm,
  mid/.style={fill=white,inner sep=2.5pt},
}

\usepackage{lmodern}
\usepackage{tikz-cd}
\usepackage{amsthm, amssymb, amsfonts}

\usepackage{amsthm, amssymb, amsfonts}	% Standard maths packages.
\usepackage[all]{xy}
\usepackage{graphicx,caption,subcaption}
\usepackage{braket}  % For nice sets
\usepackage{microtype}
\usepackage[makeroom]{cancel}
\usepackage[%
  bookmarks=true,			% show bookmarks bar?
  unicode=true,			% non-Latin characters in Acrobat’s bookmarks
  pdftitle={Invariant generalized complex structures on real flags},		%
  pdfauthor={San Martin}{Valencia}{Varea},	% 
  pdfkeywords={Invariant generalized geometry},	% list of keywords
  colorlinks=true,		% false: boxed links; true: colored links
  linkcolor=Red,			% color of internal links
  citecolor=cite,		% color of links to bibliography
  filecolor=magenta,		% color of file links
  urlcolor=RoyalBlue			% color of external links
]{hyperref}				% One of the most useful packages I have found.
\usepackage{cleveref}

\newtheoremstyle{mydef}
  {}		% Space above environment
  {}		% Space below environment
  {}		% Body font
  {}		% Indent amount (empty = no indent, \parindent = para indent)
  {\scshape}	% theorem head font
  {. }		% Punctuation after heading
  { }		% Space after heading
  {\thmname{#1}\thmnumber{ #2}\thmnote{ #3}}	% Heading spec

\newtheorem{theorem}{Theorem}[section]
\newtheorem*{theorem*}{Theorem}
\newtheorem{proposition}[theorem]{Proposition}
\newtheorem*{proposition*}{Proposition}
\newtheorem{lemma}[theorem]{Lemma}
\newtheorem*{lemma*}{Lemma}
\newtheorem{corollary}[theorem]{Corollary}
\newtheorem*{corollary*}{Corollary}
\theoremstyle{definition}
\newtheorem{definition}[theorem]{Definition}
\newtheorem{example}[theorem]{Example}

\theoremstyle{remark}
\newtheorem{remark}[theorem]{Remark}

\newtheorem*{conjecture*}{Conjecture}
\usepackage{tikz}

\DeclareMathOperator{\Ad}{Ad}
\DeclareMathOperator{\Nij}{Nij}

\author{Fabricio Valencia and Carlos Varea}
\address{}
\date{\today}

\address{F. Valencia - Instituto de Matem\'atica e Estat\'istica, Universidade de S\~ao Paulo, Rua do Mat\~ao 1010, Cidade Universit\'aria, 05508-090 S\~ao Paulo - Brazil.
\newline
C. Varea - Departamento de Matemática, Universidade Tecnológica Federal do Paraná, Av. Alberto Carazzai 1640, Centro Cornélio Procópio, 86300-000 Paraná - Brazil.
\newline  
      \phantom{xx}
  fabricio.valencia@ime.usp.br, carlosvarea@utfpr.edu.br}

\title{Invariant generalized almost complex structures on real flag manifolds}

\begin{document}
\maketitle

\begin{abstract}
We characterize those real flag manifolds that can be endowed with invariant generalized almost complex structures. We show that no $GM_2$-maximal real flag manifolds admit integrable invariant generalized almost complex structures. We give a concrete description of the generalized complex geometry on the maximal real flags of type $B_2$, $G_2$, $A_3$, and $D_l$ with $l\geq 5$, where we prove that the space of invariant generalized almost complex structures under invariant $B$-transformations is homotopy equivalent to a torus and we classify all invariant generalized almost Hermitian structures on them.
\end{abstract}

\tableofcontents
\section{Introduction}
Generalized complex geometry is a theory recently introduced by Hitchin \cite{H} and further developed by both Gualtieri \cite{G1,G2,G3} and Cavalcanti \cite{Ca2}. This provides a unified framework where it is possible to establish a generalization of several classical geometric structures as for instance symplectic, complex, K\"ahler, Calabi--Yau, among others. It is worth mentioning that both complex and symplectic geometries are extreme cases in the theory and their detailed study has allowed to extend several classical results to the context of generalized complex geometry.

Great developments on generalized complex geometry in both, mathematics and physics, have been done in last years. For instance, some of them aim at providing applications to mathematical physics as may be viewed in:

\begin{itemize}
\item \cite{HH} where the authors gave an extension of the Hitchin–Kobayashi correspondence for quiver bundles over generalized K\"ahler manifolds, thus obtaining applications to Yang--Mills theory and analysis,
\item \cite{Gr,BLPZ} where are developed various aspects of string theory using techniques from generalized K\"ahler geometry,
\item \cite{CG2} where the authors applied their study about generalized complex geometry to attack problems related to mirror symmetry and $T$-duality, and
\item \cite{LT} where is given an extension of the notion of moment map which allows both to study Hamiltonian mechanics in a more general setting and to obtain a notion of reduction in generalized complex geometry as well as in generalized K\"ahler geometry.
\end{itemize}

The aim of this paper is to start with the study of generalized complex geometry on real flag manifolds. Examples of generalized complex structures constructed by using Lie theory can be found for instance in:

\begin{itemize}
\item \cite{CG} where the authors gave a classification of all 6-dimensional nilmanifolds admitting generalized complex structures,
\item \cite{AD} where it  was described a regular class of invariant generalized complex structures on a real semisimple Lie group,
\item \cite{BMW} where they were classified all left invariant generalized complex and K\"ahler structures on simply connected 4-dimensional Lie groups as well as studied their invariant cohomologies, and, more importantly to us,
\item \cite{VS,V,GVV} where the corresponding authors studied invariant generalized complex geometry on complex flag manifolds. Firstly, in \cite{VS} was described the set of all invariant generalized almost complex structures on a maximal complex flag manifold and were presented, in a very concrete way, the integrability conditions of each of these structures. Secondly, in \cite{V} were classified all invariant generalized complex structures on a partial flag manifold with at most four isotropy summands. Thirdly, in \cite{GVV} were classified all invariant generalized K\"ahler structures on a complex maximal flag, were described the quotient spaces of all invariant generalized complex and K\"ahler structures on a complex maximal flag up to action by invariant $B$-transformations, and was given an explicit expression for the invariant pure spinor of each of these structures. 
\end{itemize}     

A flag manifold associated to a non-compact semisimple Lie algebra $\mathfrak{g}$ is a homogeneous space $\mathbb{F}_\Theta=G/P_\Theta
$ where $G$ is a connected Lie group with Lie algebra $\mathfrak{g}$ and $P_\Theta$ is a parabolic subgroup. If $K$ is a maximal compact subgroup of $G$ and $K_\Theta=K\cap P_\Theta$ then the flag manifold $\mathbb{F}_\Theta$ can be also written as $\mathbb{F}_\Theta=K/K_\Theta$. In this paper we study the existence, integrability and geometry of those generalized almost complex structures on real flag manifolds $\mathbb{F}_\Theta$ that are invariant with respect to the isotropy representation, in the case that $\mathfrak{g}$ is a split real form of a complex simple Lie algebra. We mainly focus on the case of real maximal flags.

Unlike what happens in the case of complex flag manifolds, for real flag manifolds there is no systematic way to study invariant geometric structures. Some advances have been done for instance in:
\begin{itemize}
\item \cite{PS} where the authors provided a detailed analysis of the isotropy representations for the flag manifolds of split real forms of complex simple Lie
algebras,
\item  \cite{FBS} where it was studied the existence of $K$-invariant almost complex structures on real flag manifolds, thus obtaining as conclusion that only a few flag manifolds (associated to split real forms) admit $K$-invariant complex structures, and
\item \cite{GG} where the authors studied the existence of invariant Einstein metrics on real flag manifolds associated to simple and non-compact split real forms of complex classical Lie algebras whose isotropy representation decomposes into two or three irreducible sub-representations.
\end{itemize}
This paper is structured as follows. In Section \ref{S:2} we introduce the basic concepts of both generalized complex geometry and Lie theory that we will be using throughout the paper. In Section \ref{S:3} we develop preliminary results which will be the base for the necessary computations that we must do. They will allow us to state our main results. Using the notion of $M$-equivalence classes introduced in \cite{PS} which are defined by an equivalence relation between isotropy representations of the real flag $\mathbb{F}_\Theta$, we give a concrete description of the invariant maximal isotropic subspaces in $T_{b_\Theta} \mathbb{F}_\Theta\oplus T_{b_\Theta}^\ast \mathbb{F}_\Theta$ with $b_\Theta$ the origin of $\mathbb{F}_\Theta$. Accordingly, we can state a generalization of Theorem 1. from \cite{FBS} as follows:
\begin{theorem*}[\ref{P1}]
A real flag manifold $\mathbb{F}_\Theta=K/K_\Theta$ admits a $K$-invariant generalized almost complex structure if and only if it is a maximal flag of the type $A_3$, $B_2$, $G_2$, $C_l$ for $l$ even, $D_l$ for $l\geq 4$ or it is one of the following intermediate flags:
\begin{itemize}
\item[$-$] of type $B_3$ and $\Theta=\lbrace \lambda_1-\lambda_2,\lambda_2-\lambda_3\rbrace$,
\item[$-$] of type $C_l$ with $\Theta=\lbrace \lambda_{d}-\lambda_{d+1},\cdots, \lambda_{l-1}-\lambda_{l}, 2\lambda_{l}\rbrace$ for $d>1$ odd, and
\item[$-$] of type $D_l$ with $l=4$ and $\Theta$ being one of the sets of roots: $\lbrace \lambda_1-\lambda_2,\lambda_3-\lambda_4\rbrace$, $\lbrace \lambda_1-\lambda_2,\lambda_3+\lambda_4\rbrace$, $\lbrace \lambda_3-\lambda_4,\lambda_3+\lambda_4\rbrace$.
\end{itemize}
\end{theorem*}

We show a more explicit and useful expression for both the Courant bracket and the Nijenhuis operator at the origin $b_\Theta$ of $\mathbb{F}_\Theta$ (see Proposition \ref{CaurentB}). 

There is a special kind of maximal real flag manifolds admitting invariant generalized almost complex structures that we are very interested in. Namely, a maximal real flag manifold of those described in Theorem \ref{P1} is said to be a $GM_2$-\emph{maximal real flag} if it admits at least one $M$-equivalence class root subspace of dimension $2$. In Section \ref{S:4} we deal with the integrability of invariant generalized almost complex structures on $GM_2$-maximal real flag manifolds. By means of very explicit computations using preliminary results of Section \ref{S:3} we get that:
\begin{theorem*}[\ref{NoIntegrables}]
No $GM_2$-maximal real flag manifold admits integrable $K$-invariant generalized almost complex structures.
\end{theorem*} 
It is important to point out that the computations needed to prove this result strongly depend on the classification of $M$-classes of isotropy representations which was carried out in \cite{PS}. Thus, given the nature of the problem of integrability and the rich algebraic and analytic structure of the manifolds we are working with, we reduced such an integrability problem to perform computations using the algebraic information provided by the root systems involved in each case.

Finally, Section \ref{S:5} is devoted to give a concrete description of the generalized complex geometry on the maximal real flags $\mathbb{F}$ of type $B_2$, $G_2$, $A_3$, and $D_l$ with $l\geq 5$. For these cases, we study the effects of the action by invariant $B$-transformations on the space of invariant generalized almost complex structures. We prove that for every $M$-class root space, any generalized complex structure which is not of complex type is a $B$-transform of a structure of symplectic type. We also show that every element in the set of generalized complex structures of complex type is fixed for the action induced by $B$-transformations. Adapting these results to the general case and denoting by $\mathfrak{M}_a(\mathbb{F})$ the quotient space obtained from the set of all invariant generalized almost complex structures modulo the action by invariant $B$-transformations, we obtain:
\begin{theorem*}[\ref{BModuli2.1}]
Suppose that $\displaystyle \Pi^-=\bigcup_{j=1}^d[\alpha_j]_M$ where $d$ is the number of $M$-equivalence classes. Then
$$\mathfrak M_a (\mathbb F)= \prod_{[\alpha_j]_M\subset\Pi^-} \mathfrak{M}_{\alpha_j}(\mathbb{F})=(\mathbb{R}^\ast \cup (\mathbb{R}^\ast\times \mathbb{R}))_{\alpha_1} \times \cdots \times (\mathbb{R}^\ast \cup (\mathbb{R}^\ast\times \mathbb{R}))_{\alpha_d}.$$
In particular, $\mathfrak M_a (\mathbb F)$ admits a natural topology induced from $\mathbb{R}^{2d}$ with which it is homotopy equivalent to the $d$-torus $\mathbb{T}^d$.
\end{theorem*}
Moreover, as a consequence of this result we obtain an explicit expression for the invariant pure spinor associated to each invariant generalized almost complex structure on $\mathbb{F}$ (see Corollary \ref{PureSpinor}). We also characterize all invariant generalized almost Hermitian structures on these maximal real flags (see Proposition \ref{GKhaler3}). In particular, if $\mathfrak{G}_a(\mathbb{F})$ denotes the quotient space obtained from the set of all invariant generalized metrics modulo the action by invariant $B$-transformations we have:
\begin{proposition*}[\ref{GMetrics}]
Suppose that $\displaystyle \Pi^-=\bigcup_{j=1}^d[\alpha_j]_M$ where $d$ is the number of $M$-equivalence classes. Then
$$\mathfrak{G}_a(\mathbb{F})=\lbrace((\mathbb{R}^+)^2\times\mathbb{R}) \cup ((\mathbb{R}^-)^2\times\mathbb{R})\rbrace_{\alpha_1}\times \cdots \times \lbrace((\mathbb{R}^+)^2\times\mathbb{R}) \cup ((\mathbb{R}^-)^2\times\mathbb{R})\rbrace_{\alpha_d}.$$
\end{proposition*}

It is worth noticing that the only $2$ maximal real flag manifolds mentioned in Theorem \ref{P1} that are not $GM_2$-maximal real flags are those particular cases of type $C_4$ and $D_4$. This is because each $M$-equivalence class root subspace associated to them has dimension $4$. Such a little variation in the dimension increases a lot the difficulty when performing the computations to determine whether or not an invariant generalized almost complex structure is integrable. Therefore, due to how extensive the computations were to claim that there is no integrable invariant generalized almost complex structures on a $GM_2$-maximal real flag, we leave the study of integrability as well as other aspects related to generalized geometry on the maximal real flags of type $C_4$, $D_4$, and the intermediate real flags for a forthcoming work. It may be done by using the present work plus Remark \ref{PartialCases}. Based on the results obtained in this paper and Theorem 2. from \cite{FBS} we predict that: 
\begin{conjecture*}[I]
No maximal real flag manifold of type $C_4$ or $D_4$ admits integrable $K$-invariant generalized almost complex structures.
\end{conjecture*}

\begin{conjecture*}[II]
A real flag manifold $\mathbb{F}_\Theta=K/K_\Theta$ admits a $K$-invariant generalized complex structure if and only if it is of type $C_l$ and $\Theta=\lbrace \lambda_{d}-\lambda_{d+1},\cdots, \lambda_{l-1}-\lambda_{l}, 2\lambda_{l}\rbrace$ for $d>1$ odd.
\end{conjecture*}

\vspace{.2cm}
{\bf Acknowledgments:} We would like to express our more sincere gratitude to Luiz San Martin, Viviana del Barco, Cristi\'an Ortiz, and Sebasti\'an Herrera for their valuable comments and suggestions during several stages of the present work. Varea thanks Instituto de Matem\'atica e Estat\'istica - Universidade de S\~ao Paulo for the support provided while this work was carried out. Valencia was supported by Grant 2020/07704-7 S\~ao Paulo Research Foundation - FAPESP. Varea was supported by Grant 2020/12018-5 S\~ao Paulo Research Foundation - FAPESP.
  
\section{Preliminaries}\label{S:2}
In this section we will introduce notations and summarize the most important concepts and results from both generalized complex geometry and Lie theory which we will be using throughout the paper. For more details the reader is recommended to visit for instance the references \cite{CD,G1,G2,G3,H,K}.
\subsection{Generalized complex structures} Let $M$ be a $2n$-dimensional smooth manifold. The \emph{generalized tangent bundle} on $M$ is defined to be the vector bundle $\mathbb{T}M:=TM\oplus T^\ast M$ whose space of sections is locally identified with $\mathfrak{X}(M)\oplus \Omega^1(M)$. On this vector bundle we have a natural indefinite inner product $\langle\cdot,\cdot\rangle$ of signature $(n,n)$ that is given by
\begin{equation}\label{InnerProduct}
\langle X+\xi,Y+\eta \rangle =\dfrac{1}{2}(\xi(Y)+\eta(X)).
\end{equation}
The bundle $\bigwedge^\bullet T^\ast M$ of differential forms can be viewed as a {\it spinor bundle} for $(\mathbb{T}M,\langle\cdot,\cdot\rangle)$ where the Clifford action of an element $X+\xi\in\mathbb{T}M$ on a differential form $\varphi$ is given by
$$(X+\xi)\cdot \varphi=i_X\varphi+\xi\wedge \varphi.$$ 
It is simple to check that $(X+\xi)^2\varphi=\langle X+\xi,X+\xi\rangle\varphi$. A spinor $\varphi \in \bigwedge^\bullet T^\ast M$ is said to be {\it pure} if its null space
$$L_\varphi=\{ X+\xi\in \mathbb{T}M:\ (X+\xi)\cdot \varphi=0\},$$ 
is maximal isotropic.

At each point $p\in M$, the set $\mathfrak{so}(\mathbb{T}M_p)$ of infinitesimal orthogonal transformations of $\mathbb{T}M_p$ with respect to $\langle\cdot,\cdot\rangle_p$ is identified as a vector space with $\textnormal{End}(T_pM)\oplus \bigwedge^2T_p^\ast M\oplus\bigwedge^2T_pM$. Thus, using the exponential map to the Lie group $\textnormal{SO}(\mathbb{T}M_p)$, we get that associated to each 2-form $B\in \Omega^2(M)$ we can define an orthogonal transformation $e^{B}=\left( 
\begin{array}{cc}
1 & 0\\
B & 1
\end{array}%
\right)$ which acts on $\mathbb{T}M$ by sending $X+\xi\mapsto X+\xi+i_XB$. This transformation will be called a {\it $B$-transformation}.

Let us now consider the extension of $\langle\cdot,\cdot\rangle$ to the complexification $\mathbb{T}M\otimes \mathbb{C}$. We can define a generalized almost complex structure on $M$ in three equivalent ways:
\begin{definition}\cite{H}\label{Def1}
A \emph{generalized almost complex structure} on $M$ is determined by any of the following equivalent objects:
\begin{enumerate}
\item[$\iota.$] An automorphism $\mathbb{J}:\mathbb{T}M\to \mathbb{T}M$ such that $\mathbb{J}^2=-1$ and it is orthogonal with respect to the natural indefinite inner product $\langle\cdot,\cdot\rangle$ defined in \eqref{InnerProduct}.
\item[$\iota\iota.$] A subbundle $L$ of $\mathbb{T}M\otimes \mathbb{C}$ which is maximal isotropic with respect to $\langle\cdot,\cdot\rangle$ and it satisfies $L\cap \overline{L}=\{0\}$.
\item[$\iota\iota\iota.$] A line subbundle of $\bigwedge^\bullet T^\ast M\otimes \mathbb{C}$ generated at each point by a complex pure spinor $\varphi$ such that its null space satisfies $L_\varphi\cap \overline{L}_\varphi=\{0\}$.
\end{enumerate}
\end{definition}
On the one hand, if $\mathbb{J}$ is a generalized almost complex structure on $M$ then it induces a maximal isotropic subbundle in $\mathbb{T}M\otimes \mathbb{C}$ which is given by its associated $+i$-eigenbundle. On the other hand, at each point, a pure spinor $\varphi\in \mathbb{T}M\otimes \mathbb{C}$ must have the form
 $$\varphi=e^{B+i\omega}\theta_1\wedge\cdots\wedge\theta_k,$$
  where $B$ and $\omega$ are the real and imaginary parts of a complex $2$-form and $(\theta_1,\cdots,\theta_k)$ are linearly independent complex 1-forms spanning $\Delta^\circ$ where $\Delta=\pi_1(L)$; see \cite{C}. Here $\pi_1$ denotes the projection from $\mathbb{T}M\otimes\mathbb{C}$ onto $TM\otimes\mathbb{C}$. If we denote by $\Omega=\theta_1\wedge\cdots\wedge\theta_k$, then the requirement $L_\varphi\cap \overline{L}_\varphi=\{0\}$ from part $\iota\iota\iota.$ of Definition \ref{Def1} is equivalent to asking 
  $$(\varphi,\overline{\varphi})=\Omega\wedge \overline{\Omega}\wedge\omega^{n-k}\neq 0.$$
  
We can also speak about the notion of integrability of a generalized almost complex structure in four equivalent ways. For this purpose we introduce the \emph{Courant bracket} on sections of $\mathbb{T}M$ which is defined as
$$[X+\xi,Y+\eta]=[X,Y]+\mathcal{L}_X\eta -\mathcal{L}_Y\xi-\dfrac{1}{2}d(i_X\eta-i_Y\xi).$$
Assuming the notation used in Definition \ref{Def1} we have:
\begin{definition}\cite{H,G1}\label{Def2}
A generalized almost complex structure on $M$ is said to be \emph{integrable} if any of the following equivalent facts occurs:
\begin{enumerate}
\item[$\iota.$] The Nijenhuis tensor $N_\mathbb{J}$ of $\mathbb{J}$ with respect to the Courant bracket vanishes:
$$N_\mathbb{J}(A,B)=[\mathbb{J}A,\mathbb{J}B]-[A,B]-\mathbb{J}[\mathbb{J}A,B]-\mathbb{J}[A,\mathbb{J}B]=0.$$
\item[$\iota\iota.$] The maximal isotropic subbundle $L$ of $\mathbb{T}M\otimes \mathbb{C}$ is involutive with respect to the Courant bracket.
\item[$\iota\iota\iota.$] The Nijenhuis operator $\textnormal{Nij}$ restricted to the maximal isotropic subbundle $L$ of $\mathbb{T}M\otimes \mathbb{C}$ vanishes. That is $\textnormal{Nij}|_L=0$ where
\begin{equation}\label{NijOperator}
\textnormal{Nij}(A,B,C)=\dfrac{1}{3}(\langle[A,B],C \rangle+\langle[B,C],A \rangle+\langle[C,A],B \rangle).
\end{equation}
\item[$\iota\nu.$] There exists a section $X+\xi\in \mathfrak{X}(M)\oplus \Omega^1(M)$ such that $(X+\xi)\cdot \varphi =\textnormal{d}\varphi$.
\end{enumerate}
If this is the case, the geometric object that we obtain is called a \emph{generalized complex structure}.
\end{definition}
The following natural number, which is pointwise defined, will allow us to differentiate a generalized complex structure from another.
\begin{definition}\cite{H,G1}
At each point $p\in M$, the \emph{type} of a generalized almost complex structure $\mathbb{J}$ on $M$ is defined as
$$\textnormal{Type}(\mathbb{J})_p=\textnormal{dim}(\Delta_p^\circ)=n-\dim(\Delta_p),$$
where $\Delta_p=\pi_1(L_p)$. This is not necessarily constant along $M$. When $\textnormal{Type}(\mathbb{J})$ is locally constant we say that $\mathbb{J}$ is \emph{regular}.
\end{definition}

Using generalized complex structures we can establish a notion of generalized K\"ahler structure which generalizes the classical notion of K\"ahler manifold. Due to our purposes it is more appropriate to introduce such a notion in terms of generalized almost Hermitian structures as studied in \cite{CD}. Namely:

\begin{definition}\cite{CD}\cite{G2}\label{ke}
A {\it generalized almost Hermitian structure} on $M$ is a pair $(\mathbb{J},\mathbb{J'})$ of commuting generalized almost complex structures 
such that $G=-\mathbb{J}\mathbb{J'}$ is positive definite. If moreover both $\mathbb{J}$ and $\mathbb{J'}$ are integrable then we say that the pair $(\mathbb{J},\mathbb{J'})$ is a {\it generalized K\"ahler structure}. 
The corresponding  metric  on $\mathbb{T}M$ given by
$$G(X+\alpha, Y+\beta)=\langle \mathbb{J}(X+\alpha),\mathbb{J'}(Y+\beta)\rangle$$
is called a {\it generalized metric}.
\end{definition}

Several equivalent characterizations of a generalized K\"ahler structure can be found in \cite{CD}. One of the most important reasons for which we introduced $B$-transformations above is because of the following result:
\begin{proposition}\cite{H,G1,G2}\label{ModuliM}
Let $B$ be a 2-form on $M$.
\begin{enumerate}
\item[$\iota.$] If $\mathbb{J}$ is a generalized almost complex structure on $M$, then so do $e^{-B}\mathbb{J}e^B$ and $\textnormal{Type}(\mathbb{J})_p=\textnormal{Type}(e^{-B}\mathbb{J}e^B)_p$ for all $p\in M$. If $\mathbb{J}$ is integrable, then so do $e^{-B}\mathbb{J}e^B$ if and only if $B$ is closed.
\item[$\iota\iota.$] If $(\mathbb{J},\mathbb{J'})$ is a generalized almost Hermitian structure on $M$ with generalized metric $G$, then so do $(e^{-B}\mathbb{J}e^B,e^{-B}\mathbb{J}'e^B)$ and its generalized metric is given by $G_B:=e^{-B}Ge^B$.
\item[$\iota\iota\iota.$] Given a generalized almost Hermitian structure $(\mathbb{J},\mathbb{J'})$ on $M$
 with generalized metric $G$ there always exist a Riemannian metric $g$ on $M$ and a $2$-form $b\in \Omega^2(M)$ such that
 $$G=e^{b}\left( 
 \begin{array}{cc}
 & g^{-1}\\
 g &   
 \end{array}%
 \right)e^{-b}.$$
 In particular, the signature of any generalized metric is $(n,n)$.
\end{enumerate}
\end{proposition}
Let $\mathcal{M}_a$ denote the space of all generalized almost complex structures on $M$ and let $\mathcal{K}_a\subseteq \mathcal{M}_a\times \mathcal{M}_a$ denote the set of all generalized almost Hermitian structures on $M$. The set $\mathcal{B}:=\lbrace e^B\vert\ B\in \Omega^2(M)\rbrace$ is an Abelian group with the commutative product $e^{B_1}\cdot e^{B_2}=e^{B_1+B_2}$. By Proposition \ref{ModuliM} we have that the map $\cdot:\mathcal{B}\times \mathcal{M}_a\to \mathcal{M}_a$ given by $e^{B}\cdot \mathbb{J}:=e^{-B}\mathbb{J}e^B$ is a well defined action of $\mathcal{B}$ on $\mathcal{M}_a$. In a similar way we obtain a well defined diagonal action of $\mathcal{B}$ on $\mathcal{K}_a$. These facts motivate the following definition:
\begin{definition}\cite{GVV}\label{ModuliDef}
The \emph{moduli space} of
\begin{enumerate}
\item[$\iota.$] generalized almost complex structures on $M$ under $B$-transformations is defined as the quotient space $\mathfrak{M}_a=\mathcal{M}_a/\mathcal{B}$ determined by the natural action of $\mathcal{B}$ on $\mathcal{M}_a$; and
\item[$\iota\iota.$] generalized almost Hermitian structures on $M$ under $B$-transformations is defined as the quotient space $\mathfrak{K}_a=\mathcal{K}_a/\mathcal{B}$ determined by the diagonal action of $\mathcal{B}$ on $\mathcal{K}_a$.
\end{enumerate}
\end{definition}

These quotient spaces were well described in \cite{GVV} for the case of complex maximal flag manifolds when we consider the action by invariant $B$-transformations. It is simple to see that we may define similar quotient spaces if we consider the set of all generalized complex (K\"ahler) structures on $M$. We just need to act by the subgroup $\mathcal{B}_{cl}\subset \mathcal{B}$ determined by the closed 2-forms on $M$. However, because of our purposes, in this paper we will just deal with the cases introduced in Definition \ref{ModuliDef}.

Let us now introduce some basic but instructive examples.

\begin{example}[Symplectic type $k=0$]\label{ExampleS}
	Given an {\it almost symplectic} manifold $(M,\omega)$ (that is, $\omega$ is a nondegenerate 2-form on $M$, not necessarily closed) we get that
	$$\mathbb{J}_\omega=\left( 
	\begin{array}{cc}
	0 & -\omega^{-1}\\
	\omega & 0
	\end{array}%
	\right),$$
	defines a generalized almost complex structure on $M$. Such a structure is integrable if and only if $\omega$ is closed, in which case
	we say that this is a generalized complex structure  of  {\it symplectic type}. We have that $\mathbb{J}_\omega$ determines a maximal isotropic subbundle $L_\omega=\lbrace X-i\omega(X):\ X\in TM\otimes \mathbb{C}\rbrace$ and a pure spinor line generated by $\varphi=e^{i\omega}$. This is a regular generalized almost complex structure of type $k=0$. We may transform this example by a $B$-transformation and obtain another generalized almost complex structure of type $k=0$ as follows:
	$$e^{-B}\mathbb{J}_\omega e^B=\left( 
	\begin{array}{cc}
	-\omega^{-1}B & -\omega^{-1}\\
	\omega+B\omega^{-1}B& B\omega^{-1}
	\end{array}%
	\right).$$
This has maximal isotropic subbundle $e^{-B}(L_\omega)=\lbrace X-(B+i\omega)(X):\ X\in TM\otimes \mathbb{C}\rbrace$ and pure spinor $e^{B}\cdot\varphi=e^{B+i\omega}$.
\end{example}

\begin{example}[Complex type $k=n$]\label{ExampleC}
	If $(M,J)$ is an almost complex manifold, then we have that 
	$$\mathbb{J}_c=\left( 
	\begin{array}{cc}
	-J & 0\\
	0 & J^\ast
	\end{array}%
	\right),$$
	defines a generalized almost complex structure on $M$, which  is integrable if and only if $J$ is integrable in the classical sense.
	We refer to such a generalized complex structure as being of  {\it complex type}. This is a regular generalized almost complex structure of 
	type $k=n$. We have that $\mathbb{J}_c$ determines a maximal isotropic subbundle $L=TM_{0,1}\oplus T^\ast M_{1,0} \subset \mathbb TM\otimes \mathbb C$ where $TM_{1,0}=\overline{TM_{0,1}}$ is the $+i$-eigenspace of $J$ and a pure spinor line generated by $\varphi=\Omega^{n,0}$ where $\Omega^{n,0}$ is any generator of the $(n,0)$-forms for the almost complex manifold $(M,J)$. This example may be transformed by a $B$-transformation
	$$e^{-B}\mathbb{J}_ce^B=\left( 
	\begin{array}{cc}
	-J & 0\\
	BJ+J^\ast B& J^\ast
	\end{array}%
	\right).$$
The obtained structure has maximal isotropic subbundle $e^{-B}(L)=\lbrace X+\xi-i_XB:\ X+\xi\in TM_{0,1}\oplus T^\ast M_{1,0}\rbrace$ and pure spinor $e^{B}\cdot\varphi=e^B\Omega^{n,0}$.
\end{example}

\begin{example}\label{typical}
	Assume that $(M,J,\omega)$ is a K\"ahler manifold satisfying the conditions of both Examples \ref{ExampleS} and \ref{ExampleC}.
	This means that $J$ is a complex structure on $M$ and $\omega$ is a symplectic form of type $(1,1)$, verifying $\omega J=-J^\ast \omega$,
	which implies that $\mathbb{J}_c\mathbb{J}_\omega=\mathbb{J}_\omega\mathbb{J}_c$. It is simple to check that
	$$G=-\mathbb{J}_c\mathbb{J}_\omega=\left(\begin{array}{cc}
	0 & g^{-1}\\
	g & 0
	\end{array}%
	\right),$$
	defines a generalized  metric on $\mathbb{T} M$, 
	thus obtaining that $(\mathbb{J}_c,\mathbb{J}_\omega)$ is a generalized K\"ahler structure on $M$. 
	Here $g$ denotes the Riemannian metric on $M$ induced by the formula $g(X,Y)=\omega(X,JY)$.
\end{example}

\subsection{Real flag manifolds}
We introduce real flag manifolds by following the nice and brief exposition provided in \cite{FBS}. Let us assume that $\mathfrak{g}$ is the split real form of a complex semisimple Lie algebra $\mathfrak{g}_\mathbb{C}:=\mathfrak{g}\otimes_\mathbb{R}\mathbb{C}$. If $\mathfrak{g}=\mathfrak{k}\oplus\mathfrak{a}\oplus\mathfrak{n}$ is an Iwasawa decomposition, then we get that $\mathfrak{a}$ is a Cartan subalgebra of $\mathfrak{g}$. Let $\Pi$ denote the root system associated to the pair $(\mathfrak{g},\mathfrak{a})$. If $\alpha\in\mathfrak{a}^\ast$ is a root, then we write 
$$\mathfrak{g}_\alpha=\lbrace X\in\mathfrak{g}\vert\ \textnormal{ad}(H)(X)=\alpha(H)X,\ \forall H\in\mathfrak{a}\rbrace,$$
for its corresponding root space, which is $1$-dimensional since $\mathfrak{g}$ is split. Let $\Pi^+$ be a choice of positive roots with corresponding set of positive simple roots $\Sigma$. The set of parabolic subalgebras of $\mathfrak{g}$ is parametrized by the subsets $\Theta$ of $\Sigma$. Namely, given $\Theta\subset \Sigma$, the corresponding parabolic subalgebra is given by
$$\mathfrak{p}_\Theta=\mathfrak{a}\oplus \sum_{\alpha\in \Pi^+}\mathfrak{g}_\alpha\oplus\sum_{\alpha\in \langle\Theta\rangle^{-}}\mathfrak{g}_\alpha=\mathfrak{a}\oplus \sum_{\alpha\in \langle\Theta\rangle^+\cup \langle\Theta\rangle^{-}}\mathfrak{g}_\alpha\oplus\sum_{\alpha\in \Pi^+\backslash \langle\Theta\rangle^{+}}\mathfrak{g}_\alpha,$$
where $\langle \Theta \rangle^\pm$ denotes the set of positive/negative roots generated by $\Theta$.

Let $G$ denote the Lie group of inner automorphisms of $\mathfrak{g}$ which is connected and generated by $\textnormal{Exp}\ \textnormal{ad}(\mathfrak{g})$ inside $\textnormal{GL}(\mathfrak{g})$. Let $K$ be the maximal compact subgroup of $G$. This is generated by $\textnormal{ad}(\mathfrak{k})$. The standard parabolic subgroup $P_\Theta$ of $G$ associated to $\Theta$ is the normalizer of $\mathfrak{p}_\Theta$ in $G$ and its associated \emph{flag manifold} is defined as the homogeneous space $\mathbb{F}_\Theta:=G/P_\Theta$. Given that $K$ acts transitively on $\mathbb{F}_\Theta$ we may also identify $\mathbb{F}_\Theta=K/K_\Theta$ where $K_\Theta=P_\Theta\cap K$. Fixing an arbitrary origin $b_\Theta$ in $\mathbb{F}_\Theta$, we identity the tangent space $T_{b_\Theta}\mathbb{F}_\Theta$ with the nilpotent Lie algebra
$$\mathfrak{n}_\Theta^{-}=\sum_{\alpha\in \Pi^-\backslash\langle\Theta\rangle^{-}}\mathfrak{g}_\alpha.$$

It is important to notice that under this identification the isotropy representation of $K_\Theta$ on $T_{b_\Theta}\mathbb{F}_\Theta$ is just the adjoint representation since $\mathfrak{n}_\Theta^{-}$ is normalized by $K_\Theta$. The Lie algebra $\mathfrak{k}_\Theta$ of $K_\Theta$ is given by
$$\mathfrak{k}_\Theta=\sum_{\alpha\in \langle\Theta\rangle^+\cup \langle\Theta\rangle^{-}}(\mathfrak{g}_\alpha\oplus \mathfrak{g}_{-\alpha})\cap \mathfrak{k}.$$
The compactness of $K$ implies that $\mathfrak{k}_\Theta$ has a reductive complement $\mathfrak{m}_\Theta$ so that $\mathfrak{k}=\mathfrak{k}_\Theta\oplus \mathfrak{m}_\Theta$ and  $T_{b_\Theta}\mathbb{F}_\Theta$ is also identified with $\mathfrak{m}_\Theta$. Indeed, the map $X_\alpha\mapsto X_\alpha-X_{-\alpha}$ for $\alpha\in \Pi^-\backslash\langle\Theta\rangle^{-}$ is an invariant map from $\mathfrak{n}_\Theta^{-}$ to $\mathfrak{m}_\Theta$. Throughout this paper we will call \emph{isotropy representation} to the representation of $K_\Theta$ on either $\mathfrak{n}_\Theta^{-}$ or $\mathfrak{m}_\Theta$ without making any difference or special mention. In some cases, we will even use $\displaystyle \mathfrak{n}_\Theta^{+}=\sum_{\alpha\in \Pi^-\backslash\langle\Theta\rangle^{-}}\mathfrak{g}_{-\alpha}$ instead $\mathfrak{n}_\Theta^{-}$.

Let $M$ be the centralizer of $\mathfrak{a}$ in $K$. Then $K_\Theta=M\cdot (K_\Theta)_0$ where $(K_\Theta)_0$ is the connected component of the identity in $K_\Theta$. Thus $M$ acts on $T_{b_\Theta}\mathbb{F}_\Theta$ by restricting the isotropy representation of $K_\Theta$. The group $M$ is finite and acts on each root space $\mathfrak{g}_\alpha$ leaving it invariant. 
\begin{definition}\cite{PS}
Two roots $\alpha$ and $\beta$ are called $M$-\emph{equivalent}, which we will write as $\alpha\sim_M\beta$, if the representations of $M$ on $\mathfrak{g}_\alpha$ and $\mathfrak{g}_\beta$ are equivalent.
\end{definition}
See \cite{PS} for more details about the description of the $M$-equivalence classes. Let $\langle \cdot,\cdot \rangle$ be the Cartan--Killing form of $\mathfrak{g}$. As the restriction of $\langle \cdot,\cdot \rangle$ to $\mathfrak{a}$ is non-degenerate, for every root $\alpha\in\mathfrak{a}^\ast$, we denote by $H_\alpha$ the unique element in $\mathfrak{a}$ such that $\alpha(\cdot)=\langle H_\alpha,\cdot\rangle$. Thus:

\begin{remark}\cite{PS}
If $\mathfrak{g}$ is a split real form of a complex semisimple Lie algebra $\mathfrak{g}_\mathbb{C}$, then $M$ is the finite abelian group given by 
$$M=\lbrace m_\gamma=\exp(\pi iH^\vee_\gamma)\vert\ \gamma\in \Pi \rbrace,$$
where $H^\vee_\gamma=\dfrac{2H_\gamma}{\langle \gamma,\gamma\rangle}$ is the co-root associated to $\gamma\in\Pi$. For the formula above, the exponential $\exp(t iH^\vee_\gamma)$ is taken in the complex subgroup $\textnormal{Aut}(\mathfrak{g}_\mathbb{C})$ but for $t=\pi$ we get $m_\gamma\in \textnormal{Aut}(\mathfrak{g})$ which acts on the root spaces $\mathfrak{g}_\alpha$ by $\textnormal{Ad}(m_\gamma)(X)=\pm X$. Moreover, two roots $\alpha$ and $\beta$ are $M$-equivalent if and only if for every $\gamma\in\Pi$ we get that
\begin{equation}\label{Mrelation}
\dfrac{2\langle \gamma,\alpha\rangle}{\langle \gamma,\gamma\rangle}\equiv  \dfrac{2\langle \gamma,\beta\rangle}{\langle \gamma,\gamma\rangle}\ \textnormal{mod}\ 2.
\end{equation}
\end{remark}

For the purposes of next sections we fix from here a Weyl basis for $\mathfrak{g}$ which amounts to taking $X_\alpha\in\mathbb{F}$ such that $\langle X_\alpha,X_{-\alpha}\rangle=1$ and $[X_\alpha,X_\beta]=m_{\alpha,\beta}X_{\alpha+\beta}$ with  $m_{\alpha,\beta}\in\mathbb{R}$, $m_{-\alpha,-\beta}=-m_{\alpha,\beta}$, and $m_{\alpha,\beta}=0$ is $\alpha+\beta$ is not a root.
\section{Invariant generalized almost complex structures}\label{S:3}
From the general theory of invariant tensors on homogeneous spaces, we get that $K$-invariant generalized almost complex structures on the flag manifold $\mathbb{F}_\Theta=K/K_\Theta$ are in one-to-one correspondence with complex structures $\mathbb{J}:T_{b_\Theta} \mathbb{F}_\Theta\oplus T_{b_\Theta}^\ast \mathbb{F}_\Theta\to T_{b_\Theta} \mathbb{F}_\Theta\oplus T_{b_\Theta}^\ast \mathbb{F}_\Theta$ such that:

\begin{enumerate}
\item[$\iota.$] $\langle \mathbb{J}(x) , \mathbb{J}(y)\rangle=\langle x,y \rangle$ for all $x,y\in T_{b_\Theta} \mathbb{F}_\Theta\oplus T_{b_\Theta}^\ast \mathbb{F}_\Theta$ where $\langle \cdot,\cdot\rangle$ is the natural indefinite inner product of signature $(n,n)$ defined on $T_{b_\Theta} \mathbb{F}_\Theta\oplus T_{b_\Theta}^\ast \mathbb{F}_\Theta$, which actually, after fixing a Weyl basis, agrees with the Cartan--Killing form; and
\item[$\iota\iota.$] $(\textnormal{Ad}\oplus \textnormal{Ad}^\ast)(g)\circ \mathbb{J}=\mathbb{J}\circ (\textnormal{Ad}\oplus \textnormal{Ad}^\ast)(g)$ for all $g\in K_\Theta=M\cdot(K_\Theta)_0$.
\end{enumerate}
Making use of the Cartan--Killing form we can identify $\mathfrak{n}_\Theta^{+}$ with the dual space $T_{b_\Theta}^\ast \mathbb{F}_\Theta$. Therefore, we are interested in determining $K_\Theta$-invariant orthogonal complex structures $\mathbb{J}:\mathfrak{n}_\Theta^{-}\oplus \mathfrak{n}_\Theta^{+}\to \mathfrak{n}_\Theta^{-}\oplus \mathfrak{n}_\Theta^{+}$. As we know that generalized almost complex structures are in relation with maximal isotropic subspaces, we should look at the invariant maximal isotropic subspaces of $\mathfrak{n}_\Theta^{-}\oplus \mathfrak{n}_\Theta^{+}$ with respect to the Cartan--Killing form $\langle \cdot,\cdot\rangle$.

\begin{remark}\label{PartialCases}
Recall that $K_\Theta=M\cdot(K_\Theta)_0$ and suppose that $\mathbb{J}:\mathfrak{n}_\Theta^{-}\oplus \mathfrak{n}_\Theta^{+}\to \mathfrak{n}_\Theta^{-}\oplus \mathfrak{n}_\Theta^{+}$ is just $M$-invariant. A straightforward computation allows us to prove that $\mathbb{J}$ is also $K_\Theta$-invariant if and only if it is $(K_\Theta)_0$-invariant, which, by the connectedness, is equivalent to saying that $(\textnormal{ad}\oplus\textnormal{ad}^\ast)(X)\circ \mathbb{J}=\mathbb{J}\circ (\textnormal{ad}\oplus\textnormal{ad}^\ast)(X)$ for all $X\in\mathfrak{k}_\Theta$.
\end{remark}
For every $\alpha\in \Pi^-\backslash\langle\Theta\rangle^{-}$ we denote its $M$-equivalence class by $[\alpha]_M$. We also denote 
$$\displaystyle V_{[\alpha]_M}=\sum_{\alpha \sim_M \beta} \mathfrak{g}_\beta\qquad\textnormal{and}\qquad\displaystyle V^\ast _{[\alpha]_M}= \sum_{\alpha \sim_M \beta} \mathfrak{g}^\ast _{\beta} =\sum_{\alpha \sim_M \beta} \mathfrak{g}_{-\beta}.$$ 
\begin{lemma}\label{decomposition}
Each $M$-invariant subspace $L$ in $\mathfrak{n}_\Theta^{-}\oplus \mathfrak{n}_\Theta^{+}$ has the form
\begin{equation*}
L=\sum_{[\alpha]_M}L\cap (V_{[\alpha]_M} \oplus V^\ast _{[\alpha]_M}).
\end{equation*}
\end{lemma}
\begin{proof}
Consider $\alpha\in\Pi^-\backslash\langle\Theta\rangle^{-}$ and $m_\gamma=\exp(\pi iH^\vee_\gamma)=e^{\pi iH^\vee_\gamma}\in M$. Given that $\textnormal{Ad}^\ast(m_\gamma)=\textnormal{Ad}(m_\gamma^{-1})$ we get that
$$(\textnormal{Ad}\oplus \textnormal{Ad}^\ast)(m_\gamma)(X_\alpha+X_{-\alpha})=\textnormal{Ad}(e^{\pi iH^\vee_\gamma})+\textnormal{Ad}(e^{-\pi iH^\vee_\gamma}).$$
Now, as we also have that $\textnormal{Ad}(e^Z)=e^{\textnormal{ad}(Z)}$ it follows that
\begin{eqnarray*}
	(\textnormal{Ad}\oplus \textnormal{Ad}^\ast)(m_\gamma)(X_\alpha+X_{-\alpha})& = & e^{i\pi \textnormal{ad}(H^\vee_\gamma)}X_\alpha+e^{-i\pi \textnormal{ad}(H^\vee_\gamma)}X_{-\alpha},\\
	& = & e^{i\pi \alpha(H^\vee_\gamma)}X_\alpha+e^{i\pi \alpha(H^\vee_\gamma)}X_{-\alpha}\\
	& = & e^{i\pi \alpha(H^\vee_\gamma)}(X_\alpha+X_{-\alpha})\in \mathfrak{g}_\alpha\oplus \mathfrak{g}_{-\alpha}.
\end{eqnarray*}
Then, $e^{i\pi \alpha(H^\vee_\gamma)}$ is an eigenvalue of $(\textrm{Ad} \oplus \textrm{Ad}^\ast )(M)$. Note that, from \cite{PS}, we know that $\alpha \sim_M \beta$ if and only if the congruence \eqref{Mrelation} holds true which is actually equivalent to asking that $e^{i\pi \alpha(H^\vee_\gamma)}=e^{i\pi \beta(H^\vee_\gamma)}$ for all $\gamma\in \Pi$. Therefore, $\displaystyle V_{[\alpha]_M} \oplus V^\ast _{[\alpha]_M}=\sum_{\alpha \sim_M \beta} \mathfrak{g}_\beta\oplus \mathfrak{g}_{-\beta}$ is a generalized eigenspace of $(\textrm{Ad} \oplus \textrm{Ad}^\ast )(M)$ associated with the eigenvalue $e^{i\pi \alpha(H^\vee_\gamma)}$. Since $L$ is invariant, it follows that $L$ is written as a sum of generalized eigenspaces of $(\textrm{Ad} \oplus \textrm{Ad}^\ast )(M)$, that is, 
$$L=\sum_{[\alpha]_M}L\cap (V_{[\alpha]_M} \oplus V^\ast _{[\alpha]_M}).$$
\end{proof}

\begin{lemma}\label{isotropic}
Let $\displaystyle L=\sum_{[\alpha]_M}L\cap (V_{[\alpha]_M} \oplus V^\ast _{[\alpha]_M})$ be an $M$-invariant subspace. Then $L$ is isotropic if and only if for each $\alpha$ we have that  
\begin{equation*}
L_{[\alpha]_M} = L\cap(V_{[\alpha]_M}\oplus V^\ast _{[\alpha]_M}),
\end{equation*} 
is isotropic. Moreover, if $L$ is maximal isotropic, then $L_{[\alpha]_M}$ is also maximal isotropic for each $\alpha$.
\end{lemma}
\begin{proof}
If $L$ is isotropic then it follows immediately that $L_{[\alpha]_M}$ is also isotropic.  Conversely, suppose that $L_{[\alpha]_M}$ is isotropic for each $\alpha$. Thus, $\langle X,Y\rangle=0$ for all $X,Y\in L_{[\alpha]_M}$. If $\alpha \nsim_M \beta$, then we have that $\langle X,Y\rangle=0$ for all $X \in V_{[\alpha]_M} \oplus V^\ast _{[\alpha]_M}$ and $Y\in V_{[\beta]_M} \oplus V^\ast _{[\beta]_M}$, because $\langle \mathfrak{g}_\alpha,\mathfrak{g}_\beta\rangle=0$ unless $\beta=-\alpha$ which implies that $L$ is isotropic. Finally, if $L$ is maximal isotropic, then the fact that $\displaystyle L = \sum_{[\alpha]_M} L_{[\alpha]_M}$ ensures that each $L_{[\alpha]_M}$ is maximal isotropic.
\end{proof}
Lemmas \ref{decomposition} and \ref{isotropic} imply that if $\displaystyle L = \sum_{[\alpha]_M} L_{[\alpha]_M}$ is an $M$-invariant maximal isotropic subspace in $\mathfrak{n}_\Theta^{-}\oplus \mathfrak{n}_\Theta^{+}$, then $\displaystyle L\otimes \mathbb{C} = \sum_{[\alpha]_M} (L_{[\alpha]_M}\otimes\mathbb{C})$ is an $M$-invariant maximal isotropic subspace of in $(\mathfrak{n}_\Theta^{-}\oplus \mathfrak{n}_\Theta^{+})\otimes \mathbb{C}$ with respect to $\langle\cdot,\cdot\rangle$ extended to the complexification. Therefore, because of the bijective correspondence between maximal isotropic subspaces and generalized complex structures, we have that if $\mathbb{J}$ is an $M$-invariant generalized almost complex structure on $\mathbb{F}_\Theta$, then 
\[
\mathbb{J} = \sum_{[\alpha]_M} \mathbb{J}_{[\alpha]_M},
\]
where $\mathbb{J}_{[\alpha]_M}$ is the restriction of $\mathbb{J}$ to the subspace $V_{[\alpha]_M} \oplus V^\ast _{[\alpha]_M}$. Here $\mathbb{J}_{[\alpha]_M}$ is a generalized complex structure on $V_{[\alpha]_M}$ with associated $M$-invariant maximal isotropic subspace $L_{[\alpha]_M}\otimes\mathbb{C}$. As consequence, we have that the dimension of $V_{[\alpha]_M}$ must be even and hence the amount of roots in the every $M$-equivalence class $[\alpha]_M$ must be even.
 
Summing up, using the description of the $M$-equivalence classes found in \cite{PS} we get:
\begin{theorem}\label{P1}
A real flag manifold $\mathbb{F}_\Theta$ admits an $M$-invariant generalized almost complex structure if and only if the amount of roots in the every $M$-equivalence class $[\alpha]_M$ is even. As consequence, the real flag manifolds $\mathbb{F}_\Theta$ admitting $K_\Theta$-invariant generalized almost complex structures are those with $\Theta$ described in Table \ref{table1}.
\end{theorem}
\begin{proof}
We only have to prove the sufficiency. Let $[\alpha]_M$ be an $M$-equivalence class such that the subspace $\displaystyle V_{[\alpha]_M}=\sum_{\alpha \sim_M \beta} \mathfrak{g}_\beta$ has even dimension. Recall that $\textnormal{Ad}(m)(X_\beta)=\pm X_\beta$ for all $X_\beta\in\mathfrak{g}_\beta$ and $m\in M$. In this equality, the sign does not change when $\beta$ runs
through an $M$-equivalence class. Given that $\textnormal{Ad}^\ast(m)=\textnormal{Ad}(m^{-1})$ we also have that $(\textnormal{Ad}\oplus \textnormal{Ad}^\ast)(m)(X_\beta+X_{-\beta})=\pm(X_\beta+X_{-\beta})$. Thus, we get that $(\textnormal{Ad}\oplus \textnormal{Ad}^\ast)(m)=\pm 1$ on $V_{[\alpha]_M} \oplus V^\ast _{[\alpha]_M}$. Therefore, all orthogonal complex structure on $V_{[\alpha]_M} \oplus V^\ast _{[\alpha]_M}$ are clearly $M$-invariant. Finally, by taking direct sum of orthogonal complex structures on the several $V_{[\alpha]_M} \oplus V^\ast _{[\alpha]_M}$ we obtain $M$-invariant generalized almost complex structures on  $\mathfrak{n}_\Theta^-=T_{b_\Theta}\mathbb{F}_\Theta$.
\end{proof}
\begin{table}
	\begin{tabular}{c|r}
		$\textnormal{Lie algebra type}$ & $\Theta$ \\
		\hline
	$A_3$	& $\emptyset$\\
	$B_2$ & $\emptyset$\\
	$B_3$ & $\lbrace \lambda_1-\lambda_2,\lambda_2-\lambda_3\rbrace$\\
	$C_4$ & $\emptyset$, $\lbrace \lambda_1-\lambda_2,\lambda_3-\lambda_4\rbrace$, $\lbrace \lambda_3-\lambda_4,2\lambda_4\rbrace$\\
	$C_l$ with $l\neq 4$ & $\emptyset$ only when $l$ is even\\
	& $\lbrace\lambda_d-\lambda_{d+1},\cdots, \lambda_{l-1}-\lambda_l,2\lambda_l\rbrace$ for $1<d\leq l-1$ with $d$ odd, for all $l$ \\
	$D_4$ & $\emptyset$, $\lbrace \lambda_1-\lambda_2,\lambda_3-\lambda_4\rbrace$, $\lbrace \lambda_1-\lambda_2,\lambda_3+\lambda_4\rbrace$, $\lbrace \lambda_3-\lambda_4,\lambda_3+\lambda_4\rbrace$\\
	& $\lbrace \lambda_1-\lambda_2,\lambda_2-\lambda_3,\lambda_3+\lambda_4\rbrace$, $\lbrace \lambda_2-\lambda_3,\lambda_3-\lambda_4,\lambda_3+\lambda_4\rbrace$\\
	$D_l$ with $l\leq 5$ & $\emptyset$, $\lbrace\lambda_d-\lambda_{d+1},\cdots,\lambda_{l-1}-\lambda_l,\lambda_{l-1}+\lambda_l \rbrace$ for $1<d\leq l-1$\\
	$G_4$ & $\emptyset$
	\end{tabular}
\caption{$M$-equivalence classes in $\Pi^-\backslash\langle\Theta\rangle^{-}$ with even elements.}\label{table1}
\end{table}
\subsection{The Courant bracket at the origin $b_\Theta$}
Recall that the Courant bracket on sections of $\mathbb{T}M$ is given by
$$[X+\xi,Y+\eta]=[X,Y]+\mathcal{L}_X\eta -\mathcal{L}_Y\xi-\dfrac{1}{2}\textnormal{d}(i_X\eta-i_Y\xi).$$
We want to describe this bracket at the origin $b_\Theta$ of $\mathbb{F}_\Theta$. Let $\langle \cdot,\cdot\rangle$ denote the Cartan--Killing form on $\mathfrak{g}$. As we said before, the tangent and the cotangent spaces of $\mathbb{F}_\Theta$ at $b_\Theta$ are respectively identified with the spaces
$$\mathfrak{n}_\Theta^{-}=\sum_{\alpha\in \Pi^-\backslash\langle\Theta\rangle^{-}}\mathfrak{g}_\alpha\qquad\textnormal{and}\qquad (\mathfrak{n}_\Theta^{-})^\ast= \sum_{\alpha\in \Pi^-\backslash\langle\Theta\rangle^{-}}\mathfrak{g}_\alpha^\ast=\sum_{\alpha\in \Pi^-\backslash\langle\Theta\rangle^{-}}\mathfrak{g}_{-\alpha}.$$
The identification of $\mathfrak{g}_\alpha^\ast$ with $\mathfrak{g}_{-\alpha}$ is obtained by using the Cartan--Killing form. Indeed, as $\langle X_\alpha, X_{-\alpha}\rangle=1$ we have that every element $X_\alpha^\ast\in\mathfrak{g}_\alpha^\ast$ can be represented as $X_\alpha^\ast=\langle X_{-\alpha},\cdot\rangle\approx X_{-\alpha}\in \mathfrak{g}_{-\alpha}$.

For each $X\in \mathfrak{n}_\Theta^{-}$, we denote its respective element in $(\mathfrak{n}_\Theta^{-})^\ast$ by $X^\ast=k^\flat(X^-):=\langle X^-,\cdot\rangle$, where if $\displaystyle X=\sum_{\alpha\in \Pi^-\backslash\langle\Theta\rangle^{-}} X_\alpha$, then $\displaystyle X^-=\sum_{\alpha\in \Pi^-\backslash\langle\Theta\rangle^{-}} X_{-\alpha}$. Therefore, for all $X,Y\in \mathfrak{n}_\Theta^{-}$ and $Z^\ast, W^\ast\in (\mathfrak{n}_\Theta^{-})^\ast$, we have:

\begin{eqnarray*}
	[X+Z^\ast,Y+W^\ast] &=& [X,Y]+\mathcal{L}_{X}W^\ast-\mathcal{L}_{Y}Z^\ast-\dfrac{1}{2}\textnormal{d}(i_{X}W^\ast-i_{Y}Z^\ast)\\
	& = & [X,Y]+\mathcal{L}_{X}k^\flat(W^-)-\mathcal{L}_{Y}k^\flat(Z^-)-\dfrac{1}{2}\textnormal{d}(i_{X}k^\flat(W^-)-i_{Y}k^\flat(Z^-))\\
	& = & [X,Y]+\textnormal{d}(i_Xk^\flat(W^-))+i_X(\textnormal{d}k^\flat(W^-))-\textnormal{d}(i_Yk^\flat(Z^-))-i_Y(\textnormal{d}k^\flat(Z^-))\\
	& - & \dfrac{1}{2}\textnormal{d}(i_{X}k^\flat(W^-))+\dfrac{1}{2}\textnormal{d}(i_{Y}k^\flat(Z^-))\\
	& = & [X,Y]+\dfrac{1}{2}\textnormal{d}(i_Xk^\flat(W^-))+i_X(\textnormal{d}k^\flat(W^-))-\dfrac{1}{2}\textnormal{d}(i_Yk^\flat(Z^-))-i_Y(\textnormal{d}k^\flat(Z^-)).
\end{eqnarray*}

Given that the Cartan--Killing form $\langle\cdot,\cdot\rangle$ induces a bi-invariant metric on $G$, we have that its Levi--Civita connection is given by
$$\nabla_{X}Y=-\nabla_{Y}X=\dfrac{1}{2}[X,Y],\quad \textnormal{for all}\quad X,Y\in\mathfrak{g},$$
and moreover, each element $X\in \mathfrak{g}$ is a Killing vector field, that  is, $\mathcal{L}_{X}\langle\cdot,\cdot\rangle=0$. In terms of the Levi--Civita connection this means:
$$\langle \nabla_YX,Z\rangle+\langle Y,\nabla_ZX\rangle=\langle [Y,X],Z\rangle+\langle Y,[Z,X]\rangle=0,\quad\textnormal{for all}\quad X,Y\in\mathfrak{g}.$$

On the one hand, $\textnormal{d}(i_Xk^\flat(W^-))=\textnormal{d}\langle W^-,X\rangle$. So, for all $A\in \mathfrak{n}_\Theta^{-}$:
$$\textnormal{d}\langle W^-,X\rangle (A)=A\cdot \langle W^-,X\rangle=\langle \nabla_AW^-,X\rangle+\langle W^-,\nabla_AX\rangle=-\langle \nabla_{W^-}A,X\rangle-\langle W^-,\nabla_XA\rangle=0.$$
On the other hand, $i_X(\textnormal{d}k^\flat(W^-))=\textnormal{d}k^\flat(W^-)(X)$. Thus, for all $A\in \mathfrak{n}_\Theta^{-}$:
\begin{eqnarray*}
	\textnormal{d}k^\flat(W^-)(X)(A) & = & \textnormal{d}k^\flat(W^-)(X,A)\\
	& = & X\cdot k^\flat(W^-)(A)-A\cdot k^\flat(W^-)(X)-k^\flat(W^-)([X,A])\\
	& = & X\cdot \langle W^-,A\rangle-A\cdot \langle W^-,X\rangle- \langle W^-,[X,A]\rangle\\
	& = & -2 \langle W^-,\nabla_XA\rangle\\
	& = & 2 \langle W^-,\nabla_AX\rangle\\
	& = & -2 \langle \nabla_{W^-}X,A\rangle\\
	& = & k^\flat([X,W^-])(A),
\end{eqnarray*}
that is, $i_X(\textnormal{d}k^\flat(W^-))=k^\flat([X,W^-])$. So, it follows that the Courant bracket reduces to
$$[X+Z^\ast,Y+W^\ast]=[X,Y]+k^\flat([X,W^-])-k^\flat([Y,Z^-]).$$

With similar computations we obtain that $k^\flat([X,W^-])=\textnormal{ad}^\ast_X(W^\ast)$ where $\textnormal{ad}^\ast$ denotes the co-adjoint representation of $\mathfrak{g}$. Therefore, we get that
\begin{equation*}
[X+Z^\ast,Y+W^\ast]=[X,Y]+\textnormal{ad}^\ast_X(W^\ast)-\textnormal{ad}^\ast_Y(Z^\ast).
\end{equation*}
%With this expression for the Courant bracket we can get a nicer expression for the Nijenhuis operator. 
%Recall that the Nijenhuis operator associated to the Courant Bracket is defined as
%$$\textnormal{Nij}(A,B,C)=\dfrac{1}{3}\left(\langle [A,B],C\rangle+\langle [B,C],A\rangle+\langle [C,A],B\rangle \right),$$
%for all sections $A,B,C\in \mathfrak{X}(M)\oplus \Omega^1(M)$. A straightforward computation allows us to get that
%\begin{equation*}
%\textnormal{Nij}(A_1+A_2^\ast,B_1+B_2^\ast,C_1+C_2^\ast)=\dfrac{1}{2}\left( -\langle B_2^-,[A_1,C_1] \rangle +\langle C_2^-,[A_1,B_1] \rangle+\langle A_2^-,[B_1,C_1] \rangle\right),
%\end{equation*}
%for all $A_1,B_1,C_1\in \mathfrak{n}_\Theta^{-}$ and $A_2^\ast,B_2^\ast,C_2^\ast\in (\mathfrak{n}_\Theta^{-})^\ast$.
Summing up,
\begin{proposition}\label{CaurentB}
At the origin $b_0$ of $\mathbb{F}$ we have that the Courant bracket takes the form
\begin{equation}\label{CourantBracket}
[X+Z^\ast,Y+W^\ast]=[X,Y]+\textnormal{ad}^\ast_X(W^\ast)-\textnormal{ad}^\ast_Y(Z^\ast),
\end{equation}
and the Nijenhuis operator
\begin{equation}\label{NijenhuisTensor}
\textnormal{Nij}(A_1+A_2^\ast,B_1+B_2^\ast,C_1+C_2^\ast) = \dfrac{1}{2}\left(\langle A_2^-,[B_1,C_1] \rangle+\langle B_2^-,[C_1,A_1] \rangle +\langle C_2^-,[A_1,B_1] \rangle\right),
\end{equation}
for all $X,Y,A_1,B_1,C_1\in \mathfrak{n}_\Theta^{-}$ and $Z^\ast, W^\ast A_2^\ast,B_2^\ast,C_2^\ast\in (\mathfrak{n}^{-}_\Theta)^\ast$.
\end{proposition}
\begin{proof}
The expression \eqref{NijenhuisTensor} for the Nijenhuis operator at the origin directly follows from replacing in \eqref{NijOperator} the expression \eqref{CourantBracket} that  we got for the Courant bracket.
\end{proof}
Using a Weyl basis for $\mathfrak{g}$ it is simple to check that the expression we obtained for the Nijenhuis operator \eqref{NijenhuisTensor} just depends on triples of roots ($\alpha,\beta,\alpha+\beta$). Thus, we have:
\begin{corollary}\label{usefulformula1}
For every triple of roots $(\alpha,\beta,\alpha+\beta)$ we have that
$$\Nij(X_\alpha,X_\beta,X_{\alpha+\beta} ^\ast)=\dfrac{1}{2}m_{\alpha,\beta},$$ 
and all other possible combination are zero unless these are cyclic permutations of the elements $X_\alpha, X_\beta, X^\ast _{\alpha+\beta}$.
\end{corollary}
\begin{proof}
	This result follows from a direct computation using formula \eqref{NijenhuisTensor}. In particular, for the nonzero cases
	\begin{eqnarray*}
		\Nij(X_\alpha,X_\beta,X_{\alpha+\beta} ^\ast) & = & \dfrac{1}{2}\langle X_{-(\alpha+\beta)},[X_\alpha,X_\beta]\rangle \\
		& = & \dfrac{1}{2} \langle X_{-(\alpha+\beta)},m_{\alpha,\beta}X_{\alpha+\beta} \rangle \\ 
		& = & \dfrac{1}{2}m_{\alpha,\beta}.
	\end{eqnarray*}
\end{proof}

It is wroth noticing that from Equation \eqref{CourantBracket} we have that the Courant bracket for the basic vectors of a Weyl basis is given by
\[
[X_{\alpha},X_{\beta}] = \left\lbrace \begin{array}{ll}
m_{\alpha,\beta}X_{\alpha+\beta}, \quad \textrm{if } \alpha+\beta \textrm{ is a root} \\
0, \quad \textrm{otherwise}
\end{array}\right.  \quad
[X_{\alpha},X^\ast _{\beta}] = \left\lbrace \begin{array}{ll}
m_{\alpha,-\beta}X^\ast _{\beta-\alpha}, \quad \textrm{if } \beta-\alpha \textrm{ is a root} \\
0, \quad \textrm{otherwise}
\end{array} \right. 
\]
and $[X^\ast _{\alpha},X^\ast _{\beta}] = 0$.
 
\section{Maximal real flag manifolds}\label{S:4}
Maximal real flag manifolds are those with $\Theta=\emptyset$. In this case the isotropy subgroup $K_\Theta$ is the centralizer of $\mathfrak{a}$ in $K$, that is, $K_\Theta=M$. Here we will drop up all the sub-index $\Theta$. Recall that $\textnormal{Ad}(m)(X_\alpha)=\pm X_\alpha$ for all $X_\alpha\in\mathfrak{g}_\alpha$ and $m\in M$. Given that $\textnormal{Ad}^\ast(m)=\textnormal{Ad}(m^{-1})$ we get that $(\textnormal{Ad}\oplus \textnormal{Ad}^\ast)(m)(X_\alpha+X_{-\alpha})=\pm(X_\alpha+X_{-\alpha})$. Therefore, invariant generalized almost complex structures on $\mathbb{F}$ are complex structures $\mathbb{J}:\mathfrak{n}^{-}\oplus \mathfrak{n}^{+}\to \mathfrak{n}^{-}\oplus \mathfrak{n}^{+}$ that are orthogonal with respect to $\langle\cdot,\cdot\rangle$. The representation matrix of $\langle\cdot,\cdot\rangle$ restricted to $\mathfrak{n}^{-}\oplus \mathfrak{n}^{+}$ with respect to a Weyl basis is 
$$Q=\left( 
\begin{array}{cc}
0 & I_d\\
I_d & 0
\end{array}%
\right),$$
where $I_d$ denotes the $d\times d$ identity matrix with $d$ the amount of elements in $\Pi^-$. Thus, we are interested in describing structures $\mathbb{J}$ such that $\mathbb{J}^2=-1$ and $\mathbb{J}^TQ\mathbb{J}=Q$. More precisely, according to what we did in the previous section, we just need to describe the structures $\mathbb{J}_{[\alpha]_M}$ on $V_{[\alpha]_M}$ such that $\mathbb{J}_{[\alpha]_M}^2=-1$ and $\mathbb{J}_{[\alpha]_M}^TQ\mathbb{J}_{[\alpha]_M}=Q$.

For the case of maximal real flag manifolds the $M$-equivalence classes of non-negative roots are:
\begin{enumerate}
\item[-] {\it Case} $A_3$:
$$\lbrace \lambda_2-\lambda_1,\lambda_4-\lambda_3\rbrace,\quad\lbrace \lambda_3-\lambda_1,\lambda_4-\lambda_2\rbrace,\quad\textnormal{and}\quad \lbrace \lambda_4-\lambda_1,\lambda_3-\lambda_2\rbrace.$$ 
\item[-] {\it Case} $B_2$:
$$\lbrace \lambda_2-\lambda_1,-\lambda_2-\lambda_1\rbrace\quad\textnormal{and}\quad \lbrace -\lambda_1,-\lambda_2\rbrace.$$ 
\item[-] {\it Case} $C_4$:
$$\lbrace \pm\lambda_2-\lambda_1,\pm\lambda_4-\lambda_3\rbrace,\quad\lbrace \pm\lambda_3-\lambda_1,\pm\lambda_4-\lambda_2\rbrace,\quad \lbrace \pm\lambda_4-\lambda_1,\pm\lambda_3-\lambda_2\rbrace,\quad \textnormal{and}$$
$$\lbrace -2\lambda_i:\ i=1,\cdots,4\rbrace.$$ 
\item[-] {\it Case} $C_l$ with $l$ even and $l\geq 6$:
$$A.\ \lbrace\pm \lambda_s-\lambda_i\rbrace,\ 1\leq i<s\leq l\quad \textnormal{and}\quad B.\ \lbrace 2\lambda_1,\cdots,2\lambda_l\rbrace.$$
\item[-] {\it Case} $D_4$:
$$\lbrace \pm\lambda_2-\lambda_1,\pm\lambda_4-\lambda_3\rbrace,\quad\lbrace \pm\lambda_3-\lambda_1,\pm\lambda_4-\lambda_1\rbrace,\quad \textnormal{and}\quad\lbrace \pm\lambda_4-\lambda_1,\pm\lambda_3-\lambda_2\rbrace.$$
\item[-] {\it Case} $D_l$ with $l\geq 5$:
$$\lbrace\pm \lambda_j-\lambda_i\rbrace,\ 1\leq i<j\leq l.$$
\item[-] {\it Case} $G_2$:
$$\lbrace -\lambda_1,-2\lambda_2-\lambda_1\rbrace,\quad\lbrace -\lambda_2-\lambda_1,-3\lambda_2-\lambda_1\rbrace\quad \textnormal{and}\qquad\quad\lbrace -\lambda_2,-3\lambda_2-2\lambda_1\rbrace.$$
\end{enumerate}
Motivated by this we set up the following definition.
\begin{definition}
A maximal real flag manifold of those described in Theorem \ref{P1} is said to be a $GM_2$-\emph{maximal real flag} if it admits at least an $M$-equivalence class root subspace of dimension $2$.
\end{definition}
Note that the only $2$ maximal real flag manifolds admitting invariant generalized almost complex structure that are not $MG_2$-maximal real flags are those particular cases of type $C_4$ and $D_4$. This is because each $M$-equivalence class root subspace associated to them has dimension $4$.
\subsection{Integrability on $GM_2$-maximal real flags}\label{Sub:Integrability}
We will start by analyzing the integrability of the invariant generalized almost complex structures on maximal flags $\mathbb{F}$ obtained from the cases $A_3$, $B_2$, $D_l$ with $l\geq 5$, and $G_2$. We will end the subsection by studying the integrability on maximal flags $\mathbb{F}$ of type $C_l$ with $l$ even and $l\geq 6$.

\begin{remark}\label{2DimensionalCases}

Following the computations made in \cite{VS}, for almost all these cases we just have two possibilities for $\mathbb{J}_{[\alpha]_M}$ since $\dim V_{[\alpha]_M}=2$ for all of them. Namely, 
$$\mathcal{J}^c_{[\alpha]_M}=\left( 
\begin{array}{cccc}
b_\alpha & \dfrac{-(1+b_\alpha^2)}{c_\alpha} & 0 & 0\\ 
c_\alpha & -b_\alpha & 0 & 0\\
0 & 0 & -b_\alpha & -c_\alpha \\
0 & 0 & \dfrac{1+b_\alpha^2}{c_\alpha} & b_\alpha
\end{array}%
\right):=\left( 
\begin{array}{cc}
-J^c & 0\\
0 & (J^c)^\ast
\end{array}%
\right)\qquad \textbf{complex type},$$
or else
$$\mathcal{J}^{nc}_{[\alpha]_M}=\left( 
\begin{array}{cccc}
a_\alpha & 0 & 0 & -x_\alpha\\ 
0 &  a_\alpha& x_\alpha & 0\\
0 & -y_\alpha & -a_\alpha & 0\\
y_\alpha & 0 & 0 & -a_\alpha
\end{array}%
\right):=\left( 
\begin{array}{cc}
\mathcal{A}_{\alpha} & \mathcal{X}_{\alpha} \\
\mathcal{Y}_{\alpha}  & -\mathcal{A}_{\alpha}
\end{array}%
\right)\qquad \textbf{noncomplex type},$$
with $a_\alpha,b_\alpha,c_\alpha,x_\alpha,y_\alpha\in\mathbb{R}$ such that $c_\alpha\neq 0$ and $a_\alpha^2=x_\alpha y_\alpha-1$.
\end{remark}

Consider the basis $(X_\alpha,X_\beta,X^\ast _\alpha,X_\beta ^\ast)$. For simplicity when performing the computations that we will do below, we forget up the root sub-index notation in the entries of the matrices introduced in Remark \ref{2DimensionalCases}. A straightforward computation allows us to get the following result:
\begin{lemma}
	Let $\mathbb{J}_{[\alpha]_M}$ be a generalized complex structure on $V_{[\alpha]_M}$ where $\dim V_{[\alpha]_M}=2$. If $\mathbb{J}_{[\alpha]_M}=\mathcal{J}^c_{[\alpha]_M}$ is of complex type, then its $+i$-eigenspace is given by
	\[
	L_c = \textrm{span} \{ (b+i)X_\alpha + cX_\beta, -cX_\alpha ^\ast +(b+i)X_\beta ^\ast \}.
	\]
	Otherwise, if $\mathbb{J}_{[\alpha]_M}=\mathcal{J}^{nc}_{[\alpha]_M}$ is of noncomplex type, then its $+i$-eigenspace is given by
	\[
	L_{nc} = \textrm{span} \{ xX_\alpha + (a-i)X_\beta ^\ast, -xX_\beta +(a-i)X_\alpha ^\ast \}.
	\]
\end{lemma}
As consequence of Corollary \ref{usefulformula1} we get:
\\
\\
\noindent {\bf Case $B_2$.} In this case the $M$-equivalence classes of roots are given by
\[
\lbrace \lambda_2-\lambda_1,-\lambda_2-\lambda_1\rbrace\quad\textnormal{and}\quad \lbrace -\lambda_1,-\lambda_2\rbrace.
\]
If we fix the simple root system $\Sigma = \{ \alpha,\beta\}$ with $\alpha = \lambda_2 - \lambda_1$ and $\beta = -\lambda_2$, then we have that the $M$-equivalence classes are $\lbrace \alpha,\alpha + 2\beta \rbrace$ and  $\lbrace \alpha+\beta,\beta\rbrace$, respectively. Thus, if $\mathbb{J}= \mathbb{J}_{[\alpha]} \oplus \mathbb{J}_{[\alpha+\beta]}$ is an invariant generalized almost complex structure then we just need to look at the following four possibilities:
\\
\\
\noindent {\bf 1.} $\mathbb{J}_{[\alpha]}$ and $\mathbb{J}_{[\alpha+\beta]}$ of complex type. In this case we have that $L = \textrm{span} \{(b_1+i)X_\alpha + c_1X_{\alpha+2\beta}, -c_1 X_\alpha ^\ast +(b_1+i)X_{\alpha+2\beta} ^\ast, (b_2+i)X_{\alpha+\beta} + c_2X_\beta, -c_2X_{\alpha+\beta} ^\ast +(b_2+i)X_\beta ^\ast \}$. Then, we get that 
\[
\Nij ((b_1+i)X_\alpha + c_1X_{\alpha+2\beta}, (b_2+i)X_{\alpha+\beta} + c_2X_\beta, -c_2X_{\alpha+\beta} ^\ast +(b_2+i)X_\beta ^\ast) = -\frac{1}{2}(b_1+i)c_2 ^2 m_{\alpha, \beta}\neq 0.
\]
\\
{\bf 2.} $\mathbb{J}_{[\alpha]}$ of complex type and $\mathbb{J}_{[\alpha+\beta]}$ of noncomplex type. In this case $L = \textrm{span} \{ (b+i)X_\alpha + cX_{\alpha+2\beta}, -cX_\alpha ^\ast +(b+i)X_{\alpha+2\beta} ^\ast, xX_{\alpha+\beta} + (a-i)X_\beta ^\ast, -xX_\beta +(a-i)X_{\alpha+\beta} ^\ast \}$. Then, we have that
\[
\Nij (-cX_\alpha ^\ast +(b+i)X_{\alpha+2\beta} ^\ast, xX_{\alpha+\beta} + (a-i)X_\beta ^\ast, -xX_\beta +(a-i)X_{\alpha+\beta} ^\ast) = -\frac{1}{2}x^2(b+i)m_{\alpha+\beta,\beta}\neq 0.
\]
\\
{\bf 3.} $\mathbb{J}_{[\alpha]}$ of noncomplex type and $\mathbb{J}_{[\alpha+\beta]}$ of complex type. Here $L = \textrm{span} \{ xX_\alpha + (a-i)X_{\alpha+2\beta} ^\ast, -xX_{\alpha+2\beta} +(a-i)X_\alpha ^\ast, (b+i)X_{\alpha+\beta} + cX_\beta, -cX_{\alpha+\beta} ^\ast +(b+i)X_\beta ^\ast \}$. Then, we obtain that 
\[
\Nij (xX_\alpha + (a-i)X_{\alpha+2\beta} ^\ast, (b+i)X_{\alpha+\beta} + cX_\beta, -cX_{\alpha+\beta} ^\ast +(b+i)X_\beta ^\ast) = -\frac{1}{2}xc^2 m_{\alpha,\beta}\neq 0.
\]
\\
{\bf 4.} $\mathbb{J}_{[\alpha]}$ and $\mathbb{J}_{[\alpha+\beta]}$ of noncomplex type. Here $L= \textrm{span} \{x_1X_\alpha + (a_1-i)X_{\alpha+2\beta} ^\ast, -x_1X_{\alpha+2\beta} +(a_1-i)X_\alpha ^\ast, x_2X_{\alpha+\beta} + (a_2-i)X_\beta ^\ast, -x_2X_\beta +(a_2-i)X_{\alpha+\beta} ^\ast \}$. Then, we have that
\[
\Nij (x_1X_\alpha + (a_1-i)X_{\alpha+2\beta} ^\ast, x_2X_{\alpha+\beta} + (a_2-i)X_\beta ^\ast, -x_2X_\beta +(a_2-i)X_{\alpha+\beta} ^\ast) = -\frac{1}{2}x_2 ^2 (a_1 - i)m_{\alpha+\beta,\beta}\neq 0.
\]
%Thus providing that there are no integrable invariant generalized almost complex structure on $B_2$.\\
\begin{proposition}
	Let $\mathbb{F}$ be the maximal flag manifold of type $B_2$. Then, the invariant generalized almost complex structures on $\mathbb{F}$ are not integrable.
	%The invariant generalized almost complex structures on $B_2$ are not integrable.
\end{proposition}

%\subsection*{Case $A_3$}
\noindent {\bf Case $A_3$.} 
The $M$-equivalence classes of roots are given by 
\[
\lbrace \lambda_1 - \lambda_2, \lambda_3 - \lambda_4\rbrace, \quad \lbrace \lambda_1-\lambda_3,\lambda_2-\lambda_4\rbrace \quad \textrm{and} \quad \lbrace\lambda_1 - \lambda_4,\lambda_2-\lambda_3\rbrace.
\]
If we fix the simple root system $\Sigma = \lbrace \alpha_1 = \lambda_1 - \lambda_2, \alpha_2 = \lambda_2-\lambda_3, \alpha_3 = \lambda_3 - \lambda_4\rbrace$, then we have that the $M$-equivalence classes are respectively given by
\[
\lbrace \alpha_1, \alpha_3\rbrace, \quad \lbrace \alpha_1+\alpha_2, \alpha_2+\alpha_3\rbrace \quad \textrm{and} \quad \lbrace \alpha_1+\alpha_2+\alpha_3, \alpha_2\rbrace.
\]
If $\mathbb{J}$ is an invariant generalized almost complex structure on $\mathbb{F}$, then we have one of the possibilities presented in Table \ref{table2}.
\begin{table}[htb!]
	\begin{tabular}{|r|c|c|c|}
		\hline & $\mathbb{J}_{[\alpha_1]}$ & $\mathbb{J}_{[\alpha_1+\alpha_2]}$ & $\mathbb{J}_{[\alpha_1+\alpha_2+\alpha_3]}$ \\ \hline
		{\bf 1.} & complex & complex & complex\\
		{\bf 2.} & complex & complex & noncomplex\\
		{\bf 3.} & complex & noncomplex & complex\\
		{\bf 4.} & noncomplex & complex & complex\\
		{\bf 5.} & complex & noncomplex & noncomplex\\
		{\bf 6.} & noncomplex & complex & noncomplex\\
		{\bf 7.} & complex & noncomplex & noncomplex\\
		{\bf 8.} & noncomplex & noncomplex & noncomplex \\ \hline
	\end{tabular}
	\caption{Possible combinations for $\mathbb{J}$ on $A_3$}\label{table2}
\end{table}

We must carefully analyze the integrability of all possibilities presented in Table \ref{table2}. 
\\
\\
\noindent {\bf 1.} This case comes from the invariant almost complex structures found in \cite{FBS}. This is not integrable.
\\
\\
\noindent {\bf 2.} Here the $+i$-eigenspace is given by $L = \textrm{span}\lbrace (b_1+i)X_{\alpha_1} +c_1 X_{\alpha_3}, -c_1X_{\alpha_1} ^\ast +(b_1 + i)X_{\alpha_3} ^\ast, (b_2+i)X_{\alpha_1+\alpha_2} +c_2 X_{\alpha_2+\alpha_3}, -c_2X_{\alpha_1+\alpha_2} ^\ast +(b_2 + i)X_{\alpha_2+\alpha_3} ^\ast, xX_{\alpha_1+\alpha_2+\alpha_3}+(a-i)X^\ast _{\alpha_2}, -xX_{\alpha_2} +(a-i)X_{\alpha_1+\alpha_2+\alpha_3} ^\ast \rbrace$. On the one hand, we have that
\begin{eqnarray*}
	\Nij ((b_1+i)X_{\alpha_1} +c_1 X_{\alpha_3},(b_2+i)X_{\alpha_1+\alpha_2} +c_2 X_{\alpha_2+\alpha_3},-xX_{\alpha_2} +(a-i)X_{\alpha_1+\alpha_2+\alpha_3} ^\ast) \\ = \frac{1}{2}(a-i)(c_2(b_1+i)m_{\alpha_1,\alpha_2+\alpha_3}+c_1(b_2+i)m_{\alpha_3,\alpha_1+\alpha_2})
\end{eqnarray*}
and
\begin{eqnarray*}
	\Nij ((b_1+i)X_{\alpha_1} +c_1 X_{\alpha_3},-xX_{\alpha_2} +(a-i)X_{\alpha_1+\alpha_2+\alpha_3} ^\ast, -c_2X_{\alpha_1+\alpha_2} ^\ast +(b_2 + i)X_{\alpha_2+\alpha_3} ^\ast) \\ = \dfrac{1}{2}x( (b_1+i)c_2 m_{\alpha_1,\alpha_2}- c_1 (b_2+i) m_{\alpha_3,\alpha_2}).
\end{eqnarray*}
On the other hand, suppose that $\mathbb{J}$ is integrable. So, the condition $\Nij |_L = 0$ implies that 
\[
\left\lbrace \begin{array}{l}
c_2(b_1+i)m_{\alpha_1,\alpha_2+\alpha_3}+c_1(b_2+i)m_{\alpha_3,\alpha_1+\alpha_2} = 0\\
(b_1+i)c_2 m_{\alpha_1,\alpha_2}- c_1 (b_2+i) m_{\alpha_3,\alpha_2} = 0.
\end{array}\right.
\]
In particular, it follows that
\begin{equation}\label{eq1}
\left\lbrace \begin{array}{l}
c_2 m_{\alpha_1,\alpha_2+\alpha_3}+c_1 m_{\alpha_3,\alpha_1+\alpha_2} = 0\\
c_2 m_{\alpha_1,\alpha_2}- c_1 m_{\alpha_3,\alpha_2} = 0
\end{array}\right. \quad\textnormal{implies}\quad \frac{m_{\alpha_1,\alpha_2}}{m_{\alpha_3,\alpha_2}} = \frac{c_1}{c_2} = -\frac{m_{\alpha_1,\alpha_2+\alpha_3}}{m_{\alpha_3,\alpha_1+\alpha_2}}.
\end{equation}
However, observe that by the Jacobi identity we have
\begin{eqnarray*}
	0 & = & [X_{\alpha_3},[X_{\alpha_1},X_{\alpha_2}]]-[[X_{\alpha_3},X_{\alpha_1}],X_{\alpha_2}]-[X_{\alpha_1},[X_{\alpha_3},X_{\alpha_2}]] \\
	& = & m_{\alpha_1,\alpha_2}[X_{\alpha_3},X_{\alpha_1+\alpha_2}]-m_{\alpha_3,\alpha_2}[X_{\alpha_1},X_{\alpha_2+\alpha_3}]\\
	& = & (m_{\alpha_1,\alpha_2}m_{\alpha_3,\alpha_1+\alpha_2}-m_{\alpha_3,\alpha_2}m_{\alpha_1,\alpha_2+\alpha_3})X_{\alpha_1+\alpha_2+\alpha_3}.
\end{eqnarray*}
Thus
\begin{equation}\label{eq2}
m_{\alpha_1,\alpha_2}m_{\alpha_3,\alpha_1+\alpha_2}-m_{\alpha_3,\alpha_2}m_{\alpha_1,\alpha_2+\alpha_3} = 0 \quad\textnormal{implies}\quad \frac{m_{\alpha_1,\alpha_2}}{m_{\alpha_3,\alpha_2}} = \frac{m_{\alpha_1,\alpha_2+\alpha_3}}{m_{\alpha_3,\alpha_1+\alpha_2}}.
\end{equation}
Note that Equation \eqref{eq2} contradicts Equation \eqref{eq1}. Therefore, $\mathbb{J}$ can not be integrable.
\\
\\
{\bf 3.} In this case the $+i$-eigenspace is $L = \textrm{span}\lbrace (b_1+i)X_{\alpha_1} +c_1 X_{\alpha_3}, -c_1X_{\alpha_1} ^\ast +(b_1 + i)X_{\alpha_3} ^\ast, xX_{\alpha_1+\alpha_2} +(a-i) X^\ast _{\alpha_2+\alpha_3}, -xX_{\alpha_2+\alpha_3} +(a- i)X_{\alpha_1+\alpha_2} ^\ast, (b_2+i)X_{\alpha_1+\alpha_2+\alpha_3}+c_2X_{\alpha_2},  -c_2 X_{\alpha_2} ^\ast +(b_2+i)X_{\alpha_1+\alpha_2+\alpha_3} ^\ast \rbrace$. Then, we have
\begin{eqnarray*}
	\Nij ((b_1+i)X_{\alpha_1} +c_1 X_{\alpha_3},(b_2+i)X_{\alpha_1+\alpha_2+\alpha_3}+c_2X_{\alpha_2},xX_{\alpha_1+\alpha_2} +(a-i) X^\ast _{\alpha_2+\alpha_3}) \\ = \frac{1}{2}c_1 c_2 (a-i)m_{\alpha_3,\alpha_2} \not= 0,
\end{eqnarray*}
which immediately implies that $\mathbb{J}$ is not integrable.
\\
\\
{\bf 4.} Here we obtain that $L = \textrm{span}\lbrace xX_{\alpha_1} +(a-i) X^\ast _{\alpha_3}, -xX_{\alpha_3} +(a-i)X_{\alpha_1} ^\ast, (b_1+i)X_{\alpha_1+\alpha_2} +c_1 X_{\alpha_2+\alpha_3}, -c_1X_{\alpha_1+\alpha_2} ^\ast +(b_1 + i)X_{\alpha_2+\alpha_3} ^\ast, (b_2+i)X_{\alpha_1+\alpha_2+\alpha_3}+c_2X_{\alpha_2}, -c_2 X_{\alpha_1+\alpha_2+\alpha_3} ^\ast +(b_2+i)X_{\alpha_2} ^\ast  \rbrace$. Thus
\begin{eqnarray*}
	\Nij ( xX_{\alpha_1} +(a-i) X^\ast _{\alpha_3},(b_1+i)X_{\alpha_1+\alpha_2} +c_1 X_{\alpha_2+\alpha_3},-c_2 X_{\alpha_1+\alpha_2+\alpha_3} ^\ast +(b_2+i)X_{\alpha_2} ^\ast) \\ = -\frac{1}{2}xc_1c_2m_{\alpha_1,\alpha_2+\alpha_3}\not= 0,
\end{eqnarray*}
and hence $\mathbb{J}$ is not integrable.
\\
\\
{\bf 5.} In this case $L = \textrm{span}\lbrace (b+i)X_{\alpha_1} +c X_{\alpha_3}, -cX_{\alpha_1} ^\ast +(b+i)X_{\alpha_3} ^\ast, x_1 X_{\alpha_1+\alpha_2} +(a_1 - i) X^\ast _{\alpha_2+\alpha_3}, -x_1X_{\alpha_2+\alpha_3} +(a_1 - i)X_{\alpha_1+\alpha_2} ^\ast, x_2 X_{\alpha_1+\alpha_2+\alpha_3}+(a_2-i)X^\ast _{\alpha_2}, -x_2 X_{\alpha_2} +(a_2-i)X_{\alpha_1+\alpha_2+\alpha_3} ^\ast  \rbrace$. Then, we have that
\begin{eqnarray*}
	\Nij ((b+i)X_{\alpha_1} +c X^\ast _{\alpha_3},x_1 X_{\alpha_1+\alpha_2} +(a_1 - i) X^\ast _{\alpha_2+\alpha_3},-x_2 X_{\alpha_2} +(a_2-i)X_{\alpha_1+\alpha_2+\alpha_3} ^\ast) \\ = \frac{1}{2}c(x_1(a_2-i)m_{\alpha_3,\alpha_1+\alpha_2}+x_2(a_1-i)m_{\alpha_3,\alpha_2}),
\end{eqnarray*}
and
\begin{eqnarray*}
	\Nij ((b+i)X_{\alpha_1} +c X^\ast _{\alpha_3},-x_1X_{\alpha_2+\alpha_3} +(a_1 - i)X_{\alpha_1+\alpha_2} ^\ast,-x_2 X_{\alpha_2} +(a_2-i)X_{\alpha_1+\alpha_2+\alpha_3} ^\ast) \\ = -\frac{1}{2}(b+i)(x_1(a_2-i)m_{\alpha_1,\alpha_2+\alpha_3}+x_2(a_1-i)m_{\alpha_1,\alpha_2}).
\end{eqnarray*}
Therefore, if $\Nij |_L = 0$ then we must obtain 
\[
\left\lbrace \begin{array}{l}
x_1(a_2-i)m_{\alpha_3,\alpha_1+\alpha_2}+x_2(a_1-i)m_{\alpha_3,\alpha_2} = 0\\
x_1(a_2-i)m_{\alpha_1,\alpha_2+\alpha_3}+x_2(a_1-i)m_{\alpha_1,\alpha_2}=0.
\end{array}\right.
\]
In particular, from the imaginary part we have
\begin{equation}\label{eq3}
\frac{m_{\alpha_1,\alpha_2}}{m_{\alpha_1,\alpha_2+\alpha_3}} = \frac{x_1}{x_2} = -\frac{m_{\alpha_3,\alpha_2}}{m_{\alpha_3,\alpha_1+\alpha_2}}.
\end{equation}
But observe that Equation \eqref{eq3} contradicts Equation \eqref{eq2} which we got by using Jacobi identity. Therefore, $\mathbb{J}$ can not be integrable. 
\\
\\
{\bf 6.} Here we have that $L = \textrm{span}\lbrace x_1 X_{\alpha_1} +(a_1-i) X^\ast _{\alpha_3}, -x_1X_{\alpha_3} +(a_1-i)X_{\alpha_1} ^\ast, (b+i) X_{\alpha_1+\alpha_2} +c X_{\alpha_2+\alpha_3}, -cX_{\alpha_1+\alpha_2} ^\ast +(b+i)X_{\alpha_2+\alpha_3} ^\ast, x_2 X_{\alpha_1+\alpha_2+\alpha_3}+(a_2-i)X^\ast _{\alpha_2}, -x_2 X_{\alpha_2} +(a_2-i)X_{\alpha_1+\alpha_2+\alpha_3} ^\ast  \rbrace$. Thus
\begin{eqnarray*}
	\Nij (x_1 X_{\alpha_1} +(a_1-i) X^\ast _{\alpha_3},(b+i) X_{\alpha_1+\alpha_2} +c X_{\alpha_2+\alpha_3},-x_2 X_{\alpha_2} +(a_2-i)X_{\alpha_1+\alpha_2+\alpha_3} ^\ast) \\ = \frac{1}{2}x_1 c (a_2-i)m_{\alpha_1,\alpha_2+\alpha_3} \not= 0, 
\end{eqnarray*}
which means that $\mathbb{J}$ is not integrable.
\\
\\
{\bf 7.} This case is also not integrable because the $+i$-eigenpace is $L = \textrm{span}\lbrace x_1 X_{\alpha_1} +(a_1-i) X^\ast _{\alpha_3}, -x_1X_{\alpha_3} +(a_1-i)X_{\alpha_1} ^\ast, x_2 X_{\alpha_1+\alpha_2} +(a_2-i) X^\ast _{\alpha_2+\alpha_3}, -x_2X_{\alpha_2+\alpha_3} +(a_2-i) X_{\alpha_1+\alpha_2} ^\ast, (b+i) X_{\alpha_1+\alpha_2+\alpha_3}+cX _{\alpha_2}, -c X_{\alpha_1+\alpha_2+\alpha_3} ^\ast +(b+i)X_{\alpha_2} ^\ast  \rbrace$ and
\begin{eqnarray*}
	\Nij (x_1 X_{\alpha_1} +(a_1-i) X^\ast _{\alpha_3},-x_2X_{\alpha_2+\alpha_3} +(a_2-i) X_{\alpha_1+\alpha_2} ^\ast,-c X_{\alpha_1+\alpha_2+\alpha_3} ^\ast +(b+i)X_{\alpha_2} ^\ast) \\ =-\frac{1}{2}x_1x_2(b+i)m_{\alpha_1,\alpha_2+\alpha_3} \not= 0.
\end{eqnarray*}
\\
\\
{\bf 8.} Finally, for this case we have $L = \textrm{span}\lbrace x_1 X_{\alpha_1} +(a_1-i) X^\ast _{\alpha_3}, -x_1X_{\alpha_3} +(a_1-i)X_{\alpha_1} ^\ast, x_2 X_{\alpha_1+\alpha_2} +(a_2-i) X^\ast _{\alpha_2+\alpha_3}, -x_2X_{\alpha_2+\alpha_3} +(a_2-i) X_{\alpha_1+\alpha_2} ^\ast, x_3 X_{\alpha_1+\alpha_2+\alpha_3}+(a_3-i)X _{\alpha_2} ^\ast, -x_3 X_{\alpha_2} +(a_3-i)X_{\alpha_1+\alpha_2+\alpha_3} ^\ast  \rbrace$. Note that
\begin{eqnarray*}
	\Nij (x_1 X_{\alpha_1} +(a_1-i) X^\ast _{\alpha_3},-x_2X_{\alpha_2+\alpha_3} +(a_2-i) X_{\alpha_1+\alpha_2} ^\ast,-x_3 X_{\alpha_2} +(a_3-i)X_{\alpha_1+\alpha_2+\alpha_3} ^\ast) \\ = -\frac{1}{2}x_1(x_2(a_3 - i)m_{\alpha_1,\alpha_2+\alpha_3}-x_3(a_2-i)m_{\alpha_1,\alpha_2}),
\end{eqnarray*}
and
\begin{eqnarray*}
	\Nij (-x_1X_{\alpha_3} +(a_1-i)X_{\alpha_1} ^\ast,x_2 X_{\alpha_1+\alpha_2} +(a_2-i) X^\ast _{\alpha_2+\alpha_3},-x_3 X_{\alpha_2} +(a_3-i)X_{\alpha_1+\alpha_2+\alpha_3} ^\ast )\\ = -\frac{1}{2}x_1 (x_2(a_3 - i)m_{\alpha_3,\alpha_1+\alpha_2}+x_3(a_2-i)m_{\alpha_3,\alpha_2}).
\end{eqnarray*}
So, if $\Nij|_L = 0$ then we must have
\[
\left\lbrace \begin{array}{l}
x_2(a_3 - i)m_{\alpha_1,\alpha_2+\alpha_3}-x_3(a_2-i)m_{\alpha_1,\alpha_2} = 0\\
x_2(a_3 - i)m_{\alpha_3,\alpha_1+\alpha_2}+x_3(a_2-i)m_{\alpha_3,\alpha_2} = 0.
\end{array}\right.
\]
In particular, the imaginary part vanishes which gives us
\begin{equation}\label{eq4}
\frac{m_{\alpha_1,\alpha_2}}{m_{\alpha_1,\alpha_2+\alpha_3}} = \frac{x_2}{x_3} = -\frac{m_{\alpha_3,\alpha_2}}{m_{\alpha_3,\alpha_1+\alpha_2}}.
\end{equation}
Again, we have that Equation \eqref{eq4} contradicts Equation \eqref{eq2} and, therefore, $\mathbb{J}$ can not be integrable.

%In conclusion, there are no integrable invariant generalized almost complex structures on $A_3$.
\begin{proposition}
	Let $\mathbb{F}$ be the maximal flag manifold of type $A_3$. Then, the invariant generalized almost complex structures on $\mathbb{F}$ are not integrable.
	%The invariant generalized almost complex structures on $A_3$ are not integrable.
\end{proposition}

%\subsection*{Case $G_2$}
\noindent {\bf Case $G_2$.} As we did in the previous cases, here it is simple to check that the $M$-equivalence classes of roots are given by 
\[
\lbrace \alpha, \alpha+2\beta \rbrace, \quad \lbrace \alpha+\beta,\alpha+3\beta \rbrace \quad \textrm{and} \quad \lbrace \beta,2\alpha+3\beta\rbrace.
\]

Thus, as in the case of $A_3$, we have that an invariant generalized almost complex structure $\mathbb{J}$ on $\mathbb{F}$ takes one of the possibilities presented in Table \ref{table3}.
\begin{table}[htb!]
	\begin{tabular}{|r|c|c|c|}
		\hline & $\mathbb{J}_{[\alpha]}$ & $\mathbb{J}_{[\alpha+\beta]}$ & $\mathbb{J}_{[\beta]}$ \\ \hline
		{\bf 1.} & complex & complex & complex\\
		{\bf 2.} & complex & complex & noncomplex\\
		{\bf 3.} & complex & noncomplex & complex\\
		{\bf 4.} & noncomplex & complex & complex\\
		{\bf 5.} & complex & noncomplex & noncomplex\\
		{\bf 6.} & noncomplex & complex & noncomplex\\
		{\bf 7.} & complex & noncomplex & noncomplex\\
		{\bf 8.} & noncomplex & noncomplex & noncomplex \\ \hline
	\end{tabular}
	\caption{Possible combinations for $\mathbb{J}$ on $G_2$}\label{table3}
\end{table}
\\
\\
\noindent {\bf 1.} This case comes from the invariant almost complex structures found in \cite{FBS}. This is not integrable. 
\\
\\
\noindent {\bf 2.} Here the $+i$-eigenspace is given by $L = \textrm{span}\lbrace (b_1+i)X_{\alpha} +c_1 X_{\alpha+2\beta}, -c_1X_{\alpha} ^\ast +(b_1 + i)X_{\alpha+2\beta} ^\ast, (b_2+i)X_{\alpha+\beta} +c_2 X_{\alpha+3\beta}, -c_2X_{\alpha+\beta} ^\ast +(b_2 + i)X_{\alpha+3\beta} ^\ast, xX_{\beta}+(a-i)X^\ast _{2\alpha+3\beta}, -xX_{2\alpha+3\beta} +(a-i)X_{\beta} ^\ast \rbrace$. Observe that
\begin{eqnarray*}
	\Nij ((b_1+i)X_{\alpha} +c_1 X_{\alpha+2\beta}, (b_2+i)X_{\alpha+\beta} +c_2 X_{\alpha+3\beta}, xX_{\beta}+(a-i)X^\ast _{2\alpha+3\beta}) \\ = \frac{1}{2}(a-i)((b_1+i)c_2m_{\alpha,\alpha+3\beta}+(b_2+i)c_1m_{\alpha+2\beta,\alpha+\beta}),
\end{eqnarray*}
and
\begin{eqnarray*}
	\Nij ((b_1+i)X_{\alpha} +c_1 X_{\alpha+2\beta}, xX_{\beta}+(a-i)X^\ast _{2\alpha+3\beta}, -c_2X_{\alpha+\beta} ^\ast +(b_2 + i)X_{\alpha+3\beta} ^\ast) \\ = - \frac{1}{2}x((b_1+i)c_2m_{\alpha,\beta}-(b_2+i)c_1m_{\alpha+2\beta,\beta}).
\end{eqnarray*}
Suppose that $\mathbb{J}$ is integrable, then the condition $\Nij|_L = 0$ implies that
\[
\left\lbrace \begin{array}{l}
(b_1+i)c_2m_{\alpha,\alpha+3\beta}+(b_2+i)c_1m_{\alpha+2\beta,\alpha+\beta}=0\\
(b_1+i)c_2m_{\alpha,\beta}-(b_2+i)c_1m_{\alpha+2\beta,\beta} = 0.
\end{array}\right.
\]
In particular, we have
\begin{equation}\label{g2-1}
\left\lbrace \begin{array}{l}
c_2m_{\alpha,\alpha+3\beta}+c_1m_{\alpha+2\beta,\alpha+\beta}=0\\
c_2m_{\alpha,\beta}-c_1m_{\alpha+2\beta,\beta} = 0
\end{array}\right. \quad \textnormal{implies} \quad \frac{m_{\alpha,\beta}}{m_{\alpha+2\beta}} = \frac{c_1}{c_2} = -\frac{m_{\alpha,\alpha+3\beta}}{m_{\alpha+2\beta,\alpha+\beta}}.
\end{equation}
However, as consequence of Jacobi identity we get
\begin{eqnarray*}
	0 & = & [X_{\alpha+2\beta},[X_{\alpha},X_\beta]]-[[X_{\alpha+2\beta},X_\alpha],X_\beta]-[X_\alpha,[X_{\alpha+2\beta},X_\beta]]\\
	& = & m_{\alpha,\beta}[X_{\alpha+2\beta},X_{\alpha+\beta}]- m_{\alpha+2\beta,\beta}[X_\alpha,X_{\alpha+3\beta}]\\
	& = & (m_{\alpha,\beta}m_{{\alpha+2\beta},{\alpha+\beta}}-m_{\alpha+2\beta,\beta}m_{\alpha,\alpha+3\beta})X_{2\alpha+3\beta},
\end{eqnarray*}
which tells us that
\begin{equation}\label{g2-2}
m_{\alpha,\beta}m_{{\alpha+2\beta},{\alpha+\beta}}-m_{\alpha+2\beta,\beta}m_{\alpha,\alpha+3\beta} = 0 \quad \textnormal{implies} \quad \frac{m_{\alpha,\beta}}{m_{\alpha+2\beta}} = \frac{m_{\alpha,\alpha+3\beta}}{m_{\alpha+2\beta,\alpha+\beta}}.
\end{equation}
Note that Equation \eqref{g2-2} contradicts Equation \eqref{g2-1}, thus $\mathbb{J}$ can not be integrable.
\\
\\
{\bf 3.} In this case the $+i$-eigenspace is $L = \textrm{span}\lbrace (b_1+i)X_{\alpha} +c_1 X_{\alpha+2\beta}, -c_1X_{\alpha} ^\ast +(b_1 + i)X_{\alpha+2\beta} ^\ast, xX_{\alpha+\beta} +(a-i) X^\ast _{\alpha+3\beta}, -xX_{\alpha+3\beta} +(a- i)X_{\alpha+\beta} ^\ast, (b_2+i)X_{\beta}+c_2X_{2\alpha+3\beta},  -c_2 X_{2\alpha+3\beta} ^\ast +(b_2+i)X_{\beta} ^\ast \rbrace$. Then
\begin{eqnarray*}
	\Nij ((b_1+i)X_{\alpha} +c_1 X_{\alpha+2\beta}, xX_{\alpha+\beta} +(a-i) X^\ast _{\alpha+3\beta}, -c_2 X_{2\alpha+3\beta} ^\ast +(b_2+i)X_{\beta} ^\ast \rbrace) \\ = \frac{1}{2}c_1x(b_2+i)m_{\alpha+2\beta,\alpha+\beta}\not= 0,
\end{eqnarray*}
for which we immediately get that $\mathbb{J}$ is not integrable.
\\
\\
{\bf 4.} Here we have $L = \textrm{span}\lbrace xX_{\alpha} +(a-i) X^\ast _{\alpha+2\beta}, -xX_{\alpha+2\beta} +(a-i)X_{\alpha} ^\ast, (b_1+i)X_{\alpha+\beta} +c_1 X_{\alpha+3\beta}, -c_1X_{\alpha+\beta} ^\ast +(b_1 + i)X_{\alpha+3\beta} ^\ast, (b_2+i)X_{\beta}+c_2X_{2\alpha+3\beta}, -c_2 X_{\beta} ^\ast +(b_2+i)X_{2\alpha+3\beta} ^\ast  \rbrace$. Then
\begin{eqnarray*}
	\Nij (-xX_{\alpha+2\beta} +(a-i)X_{\alpha} ^\ast, (b_2+i)X_{\beta}+c_2X_{2\alpha+3\beta}, -c_1X_{\alpha+\beta} ^\ast +(b_1 + i)X_{\alpha+3\beta} ^\ast) \\ = -\frac{1}{2}x(b_1+i)(b_2+i)m_{\alpha+2\beta,\beta}\not=0,
\end{eqnarray*}
which means that $\mathbb{J}$ is not integrable.
\\
\\
{\bf 5.} This case is also not integrable because $L = \textrm{span}\lbrace (b+i)X_{\alpha} +c X_{\alpha+2\beta}, -cX_{\alpha} ^\ast +(b+i)X_{\alpha+2\beta} ^\ast, x_1 X_{\alpha+\beta} +(a_1 - i) X^\ast _{\alpha+3\beta}, -x_1X_{\alpha+3\beta} +(a_1 - i)X_{\alpha+\beta} ^\ast, x_2 X_{\beta}+(a_2-i)X^\ast _{2\alpha+3\beta}, -x_2 X_{2\alpha+3\beta} +(a_2-i)X_{\beta} ^\ast  \rbrace$ and
\begin{eqnarray*}
	\Nij (x_2 X_{\beta}+(a_2-i)X^\ast _{2\alpha+3\beta}, x_1 X_{\alpha+\beta} +(a_1 - i) X^\ast _{\alpha+3\beta},-cX_{\alpha} ^\ast +(b+i)X_{\alpha+2\beta} ^\ast) \\ = \frac{1}{2}x_1x_2(b_i)m_{\beta,\alpha+\beta} \not= 0.
\end{eqnarray*}
\\
\\
{\bf 6.} Now we have $L = \textrm{span}\lbrace x_1 X_{\alpha} +(a_1-i) X^\ast _{\alpha+2\beta}, -x_1X_{\alpha+2\beta} +(a_1-i)X_{\alpha} ^\ast, (b+i) X_{\alpha+\beta} +c X_{\alpha+3\beta}, -cX_{\alpha+\beta} ^\ast +(b+i)X_{\alpha+3\beta} ^\ast, x_2 X_{\beta}+(a_2-i)X^\ast _{2\alpha+3\beta}, -x_2 X_{2\alpha+3\beta} +(a_2-i)X_{\beta} ^\ast  \rbrace$, then observe that
\begin{eqnarray*}
	\Nij (-x_1X_{\alpha+2\beta} +(a_1-i)X_{\alpha} ^\ast, x_2 X_{\beta}+(a_2-i)X^\ast _{2\alpha+3\beta},-cX_{\alpha+\beta} ^\ast +(b+i)X_{\alpha+3\beta} ^\ast) \\ = -\frac{1}{2}x_1x_2(b+i)m_{\alpha+2\beta,\beta}\not= 0.
\end{eqnarray*}
Therefore $\mathbb{J}$ is not integrable.
\\
\\
{\bf 7.} Here the $+i$-eigenspace is given by $L = \textrm{span}\lbrace x_1 X_{\alpha} +(a_1-i) X^\ast _{\alpha+2\beta}, -x_1X_{\alpha+2\beta} +(a_1-i)X_{\alpha} ^\ast, x_2 X_{\alpha+\beta} +(a_2-i) X^\ast _{\alpha+3\beta}, -x_2X_{\alpha+3\beta} +(a_2-i) X_{\alpha+\beta} ^\ast, (b+i) X_{\beta}+cX _{2\alpha+3\beta}, -c X_{\beta} ^\ast +(b+i)X_{2\alpha+3\beta} ^\ast  \rbrace$.  Thus
\begin{eqnarray*}
	\Nij (x_1 X_{\alpha} +(a_1-i) X^\ast _{\alpha+2\beta}, -x_2X_{\alpha+3\beta} +(a_2-i) X_{\alpha+\beta} ^\ast, -c X_{\beta} ^\ast +(b+i)X_{2\alpha+3\beta} ^\ast) \\ = -\frac{1}{2}x_1x_2(b+i)m_{\alpha,\alpha+3\beta}\not= 0,
\end{eqnarray*}
and hence $\mathbb{J}$ is not integrable.
\\
\\
{\bf 8.} Finally, we have $L = \textrm{span}\lbrace x_1 X_{\alpha} +(a_1-i) X^\ast _{\alpha+2\beta}, -x_1X_{\alpha+2\beta} +(a_1-i)X_{\alpha} ^\ast, x_2 X_{\alpha+\beta} +(a_2-i) X^\ast _{\alpha+3\beta}, -x_2X_{\alpha+3\beta} +(a_2-i) X_{\alpha+\beta} ^\ast, x_3 X_{\beta}+(a_3-i)X _{2\alpha+3\beta} ^\ast, -x_3 X_{2\alpha+3\beta} +(a_3-i)X_{\beta} ^\ast  \rbrace$. On the one hand, note that
\begin{eqnarray*}
	\Nij (x_1 X_{\alpha} +(a_1-i) X^\ast _{\alpha+2\beta}, -x_2X_{\alpha+3\beta} +(a_2-i) X_{\alpha+\beta} ^\ast, x_3 X_{\beta}+(a_3-i)X _{2\alpha+3\beta} ^\ast) \\ -\frac{1}{2}x_1(x_2(a_3-i)m_{\alpha,\alpha+3\beta}+x_3(a_2-i)m_{\alpha,\beta})
\end{eqnarray*}
and
\begin{eqnarray*}
	\Nij (-x_1X_{\alpha+2\beta} +(a_1-i)X_{\alpha} ^\ast, x_2 X_{\alpha+\beta} +(a_2-i) X^\ast _{\alpha+3\beta}, x_3 X_{\beta}+(a_3-i)X _{2\alpha+3\beta} ^\ast) \\ = -\frac{1}{2}x_1(x_2(a_3-i)m_{\alpha+2\beta,\alpha+\beta}-x_3(a_2-i)m_{\alpha+2\beta,\beta}).
\end{eqnarray*}
On the other hand, if we assume that $\mathbb{J}$ is integrable it follows from the condition $\Nij|_L=0$ that 
\[
\left\lbrace\begin{array}{l}
x_2(a_3-i)m_{\alpha,\alpha+3\beta}+x_3(a_2-i)m_{\alpha,\beta} = 0\\
x_2(a_3-i)m_{\alpha+2\beta,\alpha+\beta}-x_3(a_2-i)m_{\alpha+2\beta,\beta} = 0.
\end{array}\right.
\]
In particular, from the imaginary part we have
\begin{equation}\label{g2-3}
\left\lbrace\begin{array}{l}
x_2m_{\alpha,\alpha+3\beta}+x_3m_{\alpha,\beta} = 0\\
x_2m_{\alpha+2\beta,\alpha+\beta}-x_3m_{\alpha+2\beta,\beta} = 0
\end{array}\right. \quad \textnormal{implies} \quad \frac{m_{\alpha,\beta}}{m_{\alpha,\alpha+3\beta}}= -\frac{x_2}{x_3} = -\frac{m_{\alpha+2\beta,\beta}}{m_{\alpha+2\beta,\alpha+\beta}}.
\end{equation}
But, observe that Equation \eqref{g2-3} contradicts Equation \eqref{g2-2} and hence $\mathbb{J}$ can not be integrable.

%In consequence, there are no integrable invariant generalized almost complex structures on $G_2$.
\begin{proposition}
	Let $\mathbb{F}$ be the maximal flag manifold of type $G_2$. Then, the invariant generalized almost complex structures on $\mathbb{F}$ are not integrable.
	%The invariant generalized almost complex structures on $G_2$ are not integrable.
\end{proposition}

%\subsection*{Case $D_l$, $l\geq 5$}
\noindent {\bf Case $D_l$ with $l\geq 5$.} The $M$-equivalence classes of roots are given by 
\[
\lbrace\pm \lambda_j-\lambda_i\rbrace,\ 1\leq i<j\leq l.
\]
For the conclusion that we are going to obtain here will be enough to consider the $M$-equivalence classes $\lbrace \lambda_1-\lambda_2,\lambda_1+\lambda_2\rbrace$, $\lbrace \lambda_2-\lambda_3,\lambda_2+\lambda_3\rbrace$ and $\lbrace \lambda_1-\lambda_3,\lambda_1+\lambda_3\rbrace$. If $\mathbb{J}$ is an invariant generalized almost complex structure on $\mathbb{F}$, then we have that the possibilities for $\mathbb{J}$ at these $M$-equivalence classes are given in Table \ref{table4}.
\begin{table}[htb!]
	\begin{tabular}{|r|c|c|c|}
		\hline & $\mathbb{J}_{[\lambda_1-\lambda_2]}$ & $\mathbb{J}_{[\lambda_2-\lambda_3]}$ & $\mathbb{J}_{[\lambda_1-\lambda_3]}$ \\ \hline
		{\bf 1.} & complex & complex & complex\\
		{\bf 2.} & complex & complex & noncomplex\\
		{\bf 3.} & complex & noncomplex & complex\\
		{\bf 4.} & noncomplex & complex & complex\\
		{\bf 5.} & complex & noncomplex & noncomplex\\
		{\bf 6.} & noncomplex & complex & noncomplex\\
		{\bf 7.} & complex & noncomplex & noncomplex\\
		{\bf 8.} & noncomplex & noncomplex & noncomplex \\ \hline
	\end{tabular}
	\caption{Possible combinations for $\mathbb{J}$ on $D_l$ with $l\geq 5$}\label{table4}
\end{table}

\noindent {\bf 1.} The $+i$-eigenspace is given by $L = \textrm{span}\lbrace (b_1+i)X_{\lambda_1-\lambda_2} +c_1 X_{\lambda_1+\lambda_2},-c_1X_{\lambda_1-\lambda_2} ^\ast + (b_1+i)X^\ast _{\lambda_1+\lambda_2}, (b_2+i)X_{\lambda_2-\lambda_3} +c_2X_{\lambda_2+\lambda_3}, -c_2X^\ast _{\lambda_2-\lambda_3} + (b_2+i)X_{\lambda_2+\lambda_3} ^\ast, (b_3+i)X_{\lambda_1-\lambda_3} +c_3X_{\lambda_1+\lambda_3}, -c_3X_{\lambda_1-\lambda_3} ^\ast +(b_3+i)X_{\lambda_1+\lambda_3}\rbrace$. Thus we have
\begin{eqnarray*}
	\Nij ((b_1+i)X_{\lambda_1-\lambda_2} +c_1 X_{\lambda_1+\lambda_2}, (b_2+i)X_{\lambda_2-\lambda_3} +c_2X_{\lambda_2+\lambda_3}, -c_3X_{\lambda_1-\lambda_3} ^\ast +(b_3+i)X_{\lambda_1+\lambda_3}) \\ = -\frac{1}{2}(b_1+i)(c_3(b_2+i)m_{\lambda_1-\lambda_2,\lambda_2-\lambda_3} -c_2(b_3+i)m_{\lambda_1-\lambda_2,\lambda_2+\lambda_3}),
\end{eqnarray*}
and
\begin{eqnarray*}
	\Nij (-c_1X_{\lambda_1-\lambda_2} ^\ast + (b_1+i)X^\ast _{\lambda_1+\lambda_2}, (b_2+i)X_{\lambda_2-\lambda_3} +c_2X_{\lambda_2+\lambda_3},(b_3+i)X_{\lambda_1-\lambda_3} +c_3X_{\lambda_1+\lambda_3}) \\ =\frac{1}{2}(b_1+i)(c_3(b_2+i)m_{\lambda_2-\lambda_3,\lambda_1+\lambda_3}+c_2(b_3+i)m_{\lambda_2+\lambda_3,\lambda_1-\lambda_3}).
\end{eqnarray*}
So, if $\Nij|_L = 0$ then we get
\[
\left\lbrace \begin{array}{l}
c_3(b_2+i)m_{\lambda_1-\lambda_2,\lambda_2-\lambda_3} -c_2(b_3+i)m_{\lambda_1-\lambda_2,\lambda_2+\lambda_3} = 0\\
c_3(b_2+i)m_{\lambda_2-\lambda_3,\lambda_1+\lambda_3}+c_2(b_3+i)m_{\lambda_2+\lambda_3,\lambda_1-\lambda_3} = 0.
\end{array}\right.
\]
In particular, from the imaginary part we obtain the equation
\begin{equation}\label{dl-1}
\left\lbrace \begin{array}{l}
c_3m_{\lambda_1-\lambda_2,\lambda_2-\lambda_3} -c_2m_{\lambda_1-\lambda_2,\lambda_2+\lambda_3} = 0\\
c_3m_{\lambda_2-\lambda_3,\lambda_1+\lambda_3}+c_2m_{\lambda_2+\lambda_3,\lambda_1-\lambda_3} = 0
\end{array}\right. \quad \textnormal{implies} \quad \frac{m_{\lambda_1-\lambda_2,\lambda_2+\lambda_3}}{m_{\lambda_1-\lambda_2,\lambda_2-\lambda_3}}=\frac{c_3}{c_2}= -\frac{m_{\lambda_2-\lambda_3,\lambda_1-\lambda_3}}{m_{\lambda_2-\lambda_3,\lambda_1+\lambda_3}}.
\end{equation}
However, by applying the Jacobi identity to the vectors $X_{\lambda_1-\lambda_2}$, $X_{\lambda_2-\lambda_3}$ and $X_{\lambda_2-\lambda_3}$ we obtain 
\begin{equation}\label{dl-2}
\frac{m_{\lambda_1-\lambda_2,\lambda_2+\lambda_3}}{m_{\lambda_1-\lambda_2,\lambda_2-\lambda_3}} = \frac{m_{\lambda_2-\lambda_3,\lambda_1-\lambda_3}}{m_{\lambda_2-\lambda_3,\lambda_1+\lambda_3}}.
\end{equation}
Since Equation \eqref{dl-2} contradicts Equation \eqref{dl-1}, then $\mathbb{J}$ is not integrable.
\\
\\
{\bf 2.} Here the $+i$-eigenspace is given by $L = \textrm{span}\lbrace (b_1+i)X_{\lambda_1-\lambda_2} +c_1 X_{\lambda_1+\lambda_2}, -c_1X_{\lambda_1-\lambda_2} ^\ast +(b_1 + i)X_{\lambda_1+\lambda_2} ^\ast, (b_2+i)X_{\lambda_2-\lambda_3} +c_2 X_{\lambda_2+\lambda_3}, -c_2X_{\lambda_2-\lambda_3} ^\ast +(b_2 + i)X_{\lambda_2+\lambda_3} ^\ast, xX_{\lambda_1-\lambda_3}+(a-i)X^\ast _{\lambda_1+\lambda_3}, -xX_{\lambda_1+\lambda_3} +(a-i)X_{\lambda_1-\lambda_3} ^\ast \rbrace$. Observe that
\begin{eqnarray*}
	\Nij (-c_1X_{\lambda_1-\lambda_2} ^\ast +(b_1 + i)X_{\lambda_1+\lambda_2} ^\ast, (b_2+i)X_{\lambda_2-\lambda_3} +c_2 X_{\lambda_2+\lambda_3}, xX_{\lambda_1-\lambda_3}+(a-i)X^\ast _{\lambda_1+\lambda_3}) \\ =\frac{1}{2} (b_1+i)c_2xm_{\lambda_1+\lambda_3,\lambda_1-\lambda_3} \not= 0,
\end{eqnarray*}
which immediately implies that $\mathbb{J}$ is not integrable.
\\
\\
{\bf 3.} This case is also not integrable because the $+i$-eigenspace is $L = \textrm{span}\lbrace (b_1+i)X_{\lambda_1-\lambda_2} +c_1 X_{\lambda_1+\lambda_2}, -c_1X_{\lambda_1-\lambda_2} ^\ast +(b_1 + i)X_{\lambda_1+\lambda_2} ^\ast, xX_{\lambda_2-\lambda_3} +(a-i) X^\ast _{\lambda_2+\lambda_3}, -xX_{\lambda_2+\lambda_3} +(a- i)X_{\lambda_2-\lambda_3} ^\ast, (b_2+i)X_{\lambda_1-\lambda_3}+c_2X_{\lambda_1+\lambda_3},  -c_2 X_{\lambda_1-\lambda_3} ^\ast +(b_2+i)X_{\lambda_1+\lambda_3} ^\ast \rbrace$ and we have that
\begin{eqnarray*}
	\Nij ((b_1+i)X_{\lambda_1-\lambda_2} +c_1 X_{\lambda_1+\lambda_2}, xX_{\lambda_2-\lambda_3} +(a-i) X^\ast _{\lambda_2+\lambda_3}, -c_2 X_{\lambda_1-\lambda_3} ^\ast +(b_2+i)X_{\lambda_1+\lambda_3} ^\ast) \\ = -\frac{1}{2} xc_2(b_1+i)m_{\lambda_1-\lambda_2,\lambda_2-\lambda_3}\not= 0.
\end{eqnarray*}
\\
{\bf 4.} In this case we have  $L = \textrm{span}\lbrace xX_{\lambda_1-\lambda_2} +(a-i) X^\ast _{\lambda_1+\lambda_2}, -xX_{\lambda_1+\lambda_2} +(a-i)X_{\lambda_1-\lambda_2} ^\ast, (b_1+i)X_{\lambda_2-\lambda_3} +c_1 X_{\lambda_2+\lambda_3}, -c_1X_{\lambda_2-\lambda_3} ^\ast +(b_1 + i)X_{\lambda_2+\lambda_3} ^\ast, (b_2+i)X_{\lambda_1-\lambda_3}+c_2X_{\lambda_1+\lambda_3}, -c_2 X_{\lambda_1-\lambda_3} ^\ast +(b_2+i)X_{\lambda_1+\lambda_3} ^\ast  \rbrace$. On the one hand, we get that
\begin{eqnarray*}
	\Nij (xX_{\lambda_1-\lambda_2} +(a-i) X^\ast _{\lambda_1+\lambda_2}, (b_1+i)X_{\lambda_2-\lambda_3} +c_1 X_{\lambda_2+\lambda_3}, -c_2 X_{\lambda_1-\lambda_3} ^\ast +(b_2+i)X_{\lambda_1+\lambda_3} ^\ast) \\ = -\frac{1}{2}x(c_2(b_1+i)m_{\lambda_1-\lambda_2,\lambda_2-\lambda_3}-c_1(b_2+i)m_{\lambda_1-\lambda_2,\lambda_2+\lambda_3})
\end{eqnarray*}
and
\begin{eqnarray*}
	\Nij (xX_{\lambda_1-\lambda_2} +(a-i) X^\ast _{\lambda_1+\lambda_2}, (b_1+i)X_{\lambda_2-\lambda_3} +c_1 X_{\lambda_2+\lambda_3}, (b_2+i)X_{\lambda_1-\lambda_3}+c_2X_{\lambda_1+\lambda_3}) \\ = \frac{1}{2}(a-i)(c_2(b_1+i)m_{\lambda_2-\lambda_3,\lambda_1+\lambda_3}+c_1(b_2+i)m_{\lambda_2+\lambda_3,\lambda_1-\lambda_3}).
\end{eqnarray*}
So, if we assume that $\mathbb{J}$ is integrable then the conditions $\Nij|_L=0$ implies that
\[
\left\lbrace \begin{array}{l}
c_2(b_1+i)m_{\lambda_1-\lambda_2,\lambda_2-\lambda_3}-c_1(b_2+i)m_{\lambda_1-\lambda_2,\lambda_2+\lambda_3} =0 \\
c_2(b_1+i)m_{\lambda_2-\lambda_3,\lambda_1+\lambda_3}+c_1(b_2+i)m_{\lambda_2+\lambda_3,\lambda_1-\lambda_3} = 0.
\end{array}\right.
\]
In particular, 
\begin{equation}\label{dl-3}
\left\lbrace \begin{array}{l}
c_2m_{\lambda_1-\lambda_2,\lambda_2-\lambda_3}-c_1m_{\lambda_1-\lambda_2,\lambda_2+\lambda_3} =0 \\
c_2m_{\lambda_2-\lambda_3,\lambda_1+\lambda_3}+c_1m_{\lambda_2+\lambda_3,\lambda_1-\lambda_3} = 0
\end{array}\right. \quad \textnormal{implies} \quad \frac{m_{\lambda_1-\lambda_2,\lambda_2+\lambda_3}}{m_{\lambda_1-\lambda_2,\lambda_2-\lambda_3}}=-\frac{c_2}{c_1}=-\frac{m_{\lambda_1-\lambda_3,\lambda_2+\lambda_3}}{m_{\lambda_1+\lambda_3,\lambda_2-\lambda_3}}
\end{equation}
which contradicts Equation \eqref{dl-2}. Thus, $\mathbb{J}$ can not be integrable.
\\
\\
{\bf 5.} Here we have that $L = \textrm{span}\lbrace (b+i)X_{\lambda_1-\lambda_2} +c X_{\lambda_1+\lambda_2}, -cX_{\lambda_1-\lambda_2} ^\ast +(b+i)X_{\lambda_1+\lambda_2} ^\ast, x_1 X_{\lambda_2-\lambda_3} +(a_1 - i) X^\ast _{\lambda_2+\lambda_3}, -x_1X_{\lambda_2+\lambda_3} +(a_1 - i)X_{\lambda_2-\lambda_3} ^\ast, x_2 X_{\lambda_1-\lambda_3}+(a_2-i)X^\ast _{\lambda_1+\lambda_3}, -x_2 X_{\lambda_1+\lambda_3} +(a_2-i)X_{\lambda_1-\lambda_3} ^\ast  \rbrace$. Then, we get
\begin{eqnarray*}
	\Nij ((b+i)X_{\lambda_1-\lambda_2} +c X_{\lambda_1+\lambda_2}, x_1 X_{\lambda_2-\lambda_3} +(a_1 - i) X^\ast _{\lambda_2+\lambda_3}, -x_2 X_{\lambda_1+\lambda_3} +(a_2-i)X_{\lambda_1-\lambda_3} ^\ast) \\ =\frac{1}{2}x_1(a_2-i)(b+i)m_{\lambda_1-\lambda_2,\lambda_2-\lambda_3}\not= 0,
\end{eqnarray*}
and hence  $\mathbb{J}$ is not integrable.
\\
\\
{\bf 6.} In this case $L = \textrm{span}\lbrace x_1 X_{\lambda_1-\lambda_2} +(a_1-i) X^\ast _{\lambda_1+\lambda_2}, -x_1X_{\lambda_1+\lambda_2} +(a_1-i)X_{\lambda_1-\lambda_2} ^\ast, (b+i) X_{\lambda_2-\lambda_3} +c X_{\lambda_2+\lambda_3}, -cX_{\lambda_2-\lambda_3} ^\ast +(b+i)X_{\lambda_2+\lambda_3} ^\ast, x_2 X_{\lambda_1-\lambda_3}+(a_2-i)X^\ast _{\lambda_1+\lambda_3}, -x_2 X_{\lambda_1+\lambda_3} +(a_2-i)X_{\lambda_1-\lambda_3} ^\ast  \rbrace$, then observe that
\begin{eqnarray*}
	\Nij (x_1 X_{\lambda_1-\lambda_2} +(a_1-i) X^\ast _{\lambda_1+\lambda_2}, (b+i) X_{\lambda_2-\lambda_3} +c X_{\lambda_2+\lambda_3}, -x_2 X_{\lambda_1+\lambda_3} +(a_2-i)X_{\lambda_1-\lambda_3} ^\ast) \\ = \frac{1}{2}(b+i)(x_1(a_2-i)m_{\lambda_1-\lambda_2,\lambda_2-\lambda_3}-x_2(a_1-i)m_{\lambda_2-\lambda_3,\lambda_1+\lambda_3})
\end{eqnarray*}
and
\begin{eqnarray*}
	\Nij (x_1 X_{\lambda_1-\lambda_2} +(a_1-i) X^\ast _{\lambda_1+\lambda_2}, (b+i) X_{\lambda_2-\lambda_3} +c X_{\lambda_2+\lambda_3}, x_2 X_{\lambda_1-\lambda_3}+(a_2-i)X^\ast _{\lambda_1+\lambda_3}) \\ =\frac{1}{2}c(x_1(a_2-i)m_{\lambda_1-\lambda_2,\lambda_2+\lambda_3}+x_2(a_1-i)m_{\lambda_2-\lambda_3,\lambda_1-\lambda_3}).
\end{eqnarray*}
If $\Nij|_L =0$, then  we have
\[
\left\lbrace \begin{array}{l}
x_1(a_2-i)m_{\lambda_1-\lambda_2,\lambda_2-\lambda_3}-x_2(a_1-i)m_{\lambda_2-\lambda_3,\lambda_1+\lambda_3} =0\\
x_1(a_2-i)m_{\lambda_1-\lambda_2,\lambda_2+\lambda_3}+x_2(a_1-i)m_{\lambda_2-\lambda_3,\lambda_1-\lambda_3} =0.
\end{array}\right.
\]
From the imaginary part we obtain that
\begin{equation}\label{dl-4}
\left\lbrace \begin{array}{l}
x_1m_{\lambda_1-\lambda_2,\lambda_2-\lambda_3}-x_2m_{\lambda_2-\lambda_3,\lambda_1+\lambda_3} =0\\
x_1m_{\lambda_1-\lambda_2,\lambda_2+\lambda_3}+x_2m_{\lambda_2-\lambda_3,\lambda_1-\lambda_3} =0
\end{array}\right.\quad \textnormal{implies} \quad \frac{m_{\lambda_2-\lambda_3,\lambda_1+\lambda_3}}{m_{\lambda_1-\lambda_2,\lambda_2-\lambda_3}}=\frac{x_1}{x_2}= -\frac{m_{\lambda_2+\lambda_3,\lambda_1-\lambda_3}}{m_{\lambda_1-\lambda_2,\lambda_2+\lambda_3}}
\end{equation}
fact which also contradicts Equation \eqref{dl-2} and hence $\mathbb{J}$ can not be integrable.
\\
\\
{\bf 7.} Here the $+i$-eigenspace is $L = \textrm{span}\lbrace x_1 X_{\lambda_1-\lambda_2} +(a_1-i) X^\ast _{\lambda_1+\lambda_2}, -x_1X_{\lambda_1+\lambda_2} +(a_1-i)X_{\lambda_1-\lambda_2} ^\ast, x_2 X_{\lambda_2-\lambda_3} +(a_2-i) X^\ast _{\lambda_2+\lambda_3}, -x_2X_{\lambda_2+\lambda_3} +(a_2-i) X_{\lambda_2-\lambda_3} ^\ast, (b+i) X_{\lambda_1-\lambda_3}+cX _{\lambda_1+\lambda_3}, -c X_{\lambda_1-\lambda_3} ^\ast +(b+i)X_{\lambda_1+\lambda_3} ^\ast  \rbrace$ and
\begin{eqnarray*}
	\Nij (x_1 X_{\lambda_1-\lambda_2} +(a_1-i) X^\ast _{\lambda_1+\lambda_2}, x_2 X_{\lambda_2-\lambda_3} +(a_2-i) X^\ast _{\lambda_2+\lambda_3}, -c X_{\lambda_1-\lambda_3} ^\ast +(b+i)X_{\lambda_1+\lambda_3} ^\ast) \\ =-\frac{1}{2}x_1x_2cm_{\lambda_1-\lambda_2,\lambda_2-\lambda_3}\not= 0.
\end{eqnarray*}
Therefore $\mathbb{J}$ is not integrable.
\\
\\
{\bf 8.} Finally, we have $L = \textrm{span}\lbrace x_1 X_{\lambda_1-\lambda_2} +(a_1-i) X^\ast _{\lambda_1+\lambda_2}, -x_1X_{\lambda_1+\lambda_2} +(a_1-i)X_{\lambda_1-\lambda_2} ^\ast, x_2 X_{\lambda_2-\lambda_3} +(a_2-i) X^\ast _{\lambda_2+\lambda_3}, -x_2X_{\lambda_2+\lambda_3} +(a_2-i) X_{\lambda_2-\lambda_3} ^\ast, x_3 X_{\lambda_1-\lambda_3}+(a_3-i)X _{\lambda_1+\lambda_3} ^\ast, -x_3 X_{\lambda_1+\lambda_3} +(a_3-i)X_{\lambda_1-\lambda_3} ^\ast  \rbrace$. Note that
\begin{eqnarray*}
	\Nij (x_1 X_{\lambda_1-\lambda_2} +(a_1-i) X^\ast _{\lambda_1+\lambda_2}, x_2 X_{\lambda_2-\lambda_3} +(a_2-i) X^\ast _{\lambda_2+\lambda_3}, -x_3 X_{\lambda_1+\lambda_3} +(a_3-i)X_{\lambda_1-\lambda_3} ^\ast) \\ = \frac{1}{2}x_2(x_1 (a_3-i)m_{\lambda_1-\lambda_2,\lambda_2-\lambda_3}-x_3(a_1-i)m_{\lambda_2-\lambda_3,\lambda_1+\lambda_3}),
\end{eqnarray*}
and
\begin{eqnarray*}
	\Nij (x_1 X_{\lambda_1-\lambda_2} +(a_1-i) X^\ast _{\lambda_1+\lambda_2}, -x_2X_{\lambda_2+\lambda_3} +(a_2-i) X_{\lambda_2-\lambda_3} ^\ast, x_3 X_{\lambda_1-\lambda_3}+(a_3-i)X _{\lambda_1+\lambda_3} ^\ast) \\ =-\frac{1}{2}x_2(x_1(a_3-i)m_{\lambda_1-\lambda_2,\lambda_2+\lambda_3}+x_3(a_1-i)m_{\lambda_2+\lambda_3,\lambda_1-\lambda_3}).
\end{eqnarray*}
Now, if we assume that $\mathbb{J}$ is integrable then we must have
\[
\left\lbrace \begin{array}{l}
x_1 (a_3-i)m_{\lambda_1-\lambda_2,\lambda_2-\lambda_3}-x_3(a_1-i)m_{\lambda_2-\lambda_3,\lambda_1+\lambda_3} = 0\\
x_1(a_3-i)m_{\lambda_1-\lambda_2,\lambda_2+\lambda_3}+x_3(a_1-i)m_{\lambda_2+\lambda_3,\lambda_1-\lambda_3} = 0.
\end{array}\right.
\]
From the imaginary part we get that
\begin{equation}\label{dl-5}
\left\lbrace \begin{array}{l}
x_1 m_{\lambda_1-\lambda_2,\lambda_2-\lambda_3}-x_3 m_{\lambda_2-\lambda_3,\lambda_1+\lambda_3} = 0\\
x_1 m_{\lambda_1-\lambda_2,\lambda_2+\lambda_3}+x_3 m_{\lambda_2+\lambda_3,\lambda_1-\lambda_3} = 0
\end{array}\right.\quad \textnormal{implies} \quad \frac{m_{\lambda_1-\lambda_2,\lambda_2+\lambda_3}}{m_{\lambda_2+\lambda_3,\lambda_1-\lambda_3}}=-\frac{x_3}{x_1} =-\frac{m_{\lambda_1-\lambda_2,\lambda2-\lambda_3}}{m_{\lambda_2-\lambda_3,\lambda_1+\lambda_3}}.
\end{equation}
It is simple to see that Equation \eqref{dl-5} contradicts Equation \eqref{dl-2} and hence $\mathbb{J}$ can not be integrable.
%In conclusion, there are no integrable invariant generalized almost complex structures on $D_l$ with $l\geq 5$.
\begin{proposition}
	Let $\mathbb{F}$ be the maximal flag manifold of type $D_l$ with $l\geq 5$. Then, the invariant generalized almost complex structures on $\mathbb{F}$ are not integrable.
	%The invariant generalized almost complex structures on $D_l$ with $l\geq 5$ are not integrable.
\end{proposition}

For the remaining case we will use the expression described in Equation \eqref{CourantBracket} for the Courant bracket. More specifically, we will use the fact that a generalized complex structure $\mathbb{J}$ is integrable if and only if $N_\mathbb{J} \equiv 0$, where the `Nijenhuis tensor' is given by
\begin{equation}
N_{\mathbb{J}} (A,B) = [\mathbb{J}A,\mathbb{J}B]-[A,B]-\mathbb{J}[A,\mathbb{J}B]-\mathbb{J}[\mathbb{J}A,B],
\end{equation}
with $[\cdot,\cdot]$ denoting the Courant bracket. To make the computations easier recall that the Courant bracket for the basic vectors of a Weyl basis is given by
\[
[X_{\alpha},X_{\beta}] = \left\lbrace \begin{array}{ll}
m_{\alpha,\beta}X_{\alpha+\beta}, \quad \textrm{if } \alpha+\beta \textrm{ is a root} \\
0, \quad \textrm{otherwise}
\end{array}\right.  \quad
[X_{\alpha},X^\ast _{\beta}] = \left\lbrace \begin{array}{ll}
m_{\alpha,-\beta}X^\ast _{\beta-\alpha}, \quad \textrm{if } \beta-\alpha \textrm{ is a root} \\
0, \quad \textrm{otherwise}
\end{array} \right. 
\]
and $[X^\ast _{\alpha},X^\ast _{\beta}] = 0$.\\

\noindent {\bf Case $C_l$ with $l\geq 6$ and $l$ even.} The $M$-equivalence classes are given by
\[
\lbrace \lambda_i-\lambda_s,\lambda_i+\lambda_s \rbrace,\ 1\leq i<s\leq l\quad \textnormal{and}\quad \lbrace 2\lambda_1,\cdots,2\lambda_l \rbrace.
\]
Let $\mathbb{J}$ be an invariant generalized almost complex structure on $\mathbb{F}$. Observe that $\mathbb{J}_{[\lambda_i-\lambda_s]}$ can be either of complex or noncomplex type, for all $1\leq i<s\leq l$. Therefore, we will see what happens with the integrability of $\mathbb{J}$ when analyzing $\mathbb{J}_{[\lambda_1-\lambda_2]}$ in both cases, complex and noncomplex.
\\
\\
{\bf 1.} Suppose that $\mathbb{J}_{[\lambda_1-\lambda_2]}$ is of complex type and let us write $\mathbb{J}(X_{2\lambda_2}) = \sum_{j=1} ^l c_jX_{2\lambda_j} + \sum_{j=1} ^l d_jX^\ast _{2\lambda_j}$. Assuming the notation of Remark \ref{2DimensionalCases} we have
\begin{eqnarray*}
	N_{\mathbb{J}}(X_{\lambda_1+\lambda_2},X_{2\lambda_2}) & = & \left[ -\frac{(1+b^2)}{c}X_{\lambda_1-\lambda_2}-bX_{\lambda_1+\lambda_2},\sum_{j=1} ^l c_jX_{2\lambda_j} + \sum_{j=1} ^l d_jX^\ast _{2\lambda_j}\right] \\
	& & -\mathbb{J}\left[X_{\lambda_1+\lambda_2},\sum_{j=1} ^l c_jX_{2\lambda_j} + \sum_{j=1} ^l d_jX^\ast _{2\lambda_j}\right] - \mathbb{J}\left[-\frac{(1+b^2)}{c}X_{\lambda_1-\lambda_2}-bX_{\lambda_1+\lambda_2},X_{2\lambda_2}\right]\\
	& = & -\frac{(1+b^2)}{c}c_2m_{\lambda_1-\lambda_2,2\lambda_2}X_{\lambda_1+\lambda_2}-\frac{(1+b^2)}{c}d_1m_{\lambda_1-\lambda_2,-2\lambda_1}X_{\lambda_1+\lambda_2} ^\ast \\
	& & - bd_1m_{\lambda_1+\lambda_2,-2\lambda_1}X_{\lambda_1-\lambda_2} ^\ast -\mathbb{J}\left( d_1m_{\lambda_1+\lambda_2,-2\lambda_1} X_{\lambda_1-\lambda_2} ^\ast \right) \\
	& & - \mathbb{J}\left( -\frac{(1+b^2)}{c}m_{\lambda_1-\lambda_2,2\lambda_2}X_{\lambda_1+\lambda_2}  \right)\\
	& = & -\frac{(1+b^2)}{c}c_2m_{\lambda_1-\lambda_2,2\lambda_2}X_{\lambda_1+\lambda_2}-\frac{(1+b^2)}{c}d_1m_{\lambda_1-\lambda_2,-2\lambda_1}X_{\lambda_1+\lambda_2} ^\ast \\
	& & -bd_1m_{\lambda_1+\lambda_2,-2\lambda_1}X_{\lambda_1-\lambda_2} ^\ast - d_1m_{\lambda_1+\lambda_2,-2\lambda_1} \left( -bX_{\lambda_1-\lambda_2} ^\ast + \frac{1+b^2}{c}X_{\lambda_1+\lambda_2} ^\ast \right) \\
	& & + \frac{(1+b^2)}{c}m_{\lambda_1-\lambda_2,2\lambda_2} \left( -\frac{(1+b^2)}{c} X_{\lambda_1-\lambda_2} - bX_{\lambda_1+\lambda_2}  \right)\\
	& = & -\frac{(1+b^2)}{c}m_{\lambda_1-\lambda_2,2\lambda_2}(1+b)X_{\lambda_1+\lambda_2} - 2\frac{(1+b^2)}{c}d_1m_{\lambda_1-\lambda_2,-2\lambda_1}X^\ast _{\lambda_1+\lambda_2} \\
	& & - \left( \frac{(1+b^2)}{c}\right)^2 m_{\lambda_1-\lambda_2,2\lambda_2}X_{\lambda_1-\lambda_2} \neq 0,
\end{eqnarray*}
since $\frac{(1+b^2)}{c} \neq 0$ and $m_{\lambda_1-\lambda_2,2\lambda_2}\neq 0$.
\\
\\
{\bf 1.} Now suppose $\mathbb{J}_{[\lambda_1-\lambda_2]}$ of noncomplex type. Thus,
\begin{eqnarray*}
N_{\mathbb{J}}(X_{\lambda_1-\lambda_2},X_{\lambda_1+\lambda_2}) & = & [\mathbb{J}X_{\lambda_1-\lambda_2},\mathbb{J}X_{\lambda_1+\lambda_2}]- [X_{\lambda_1-\lambda_2},X_{\lambda_1+\lambda_2}]-\mathbb{J}[X_{\lambda_1-\lambda_2},\mathbb{J}X_{\lambda_1+\lambda_2}]\\
& & -\mathbb{J}[\mathbb{J}X_{\lambda_1-\lambda_2},X_{\lambda_1+\lambda_2}]\\
& = & [aX_{\lambda_1-\lambda_2}+yX_{\lambda_1+\lambda_2} ^\ast ,aX_{\lambda_1+\lambda_2}-yX_{\lambda_1-\lambda_2} ^\ast] - m_{\lambda_1-\lambda_2,\lambda_1+\lambda_2}X_{2\lambda_1}\\
& & -\mathbb{J}[X_{\lambda_1-\lambda_2},aX_{\lambda_1+\lambda_2}-yX_{\lambda_1-\lambda_2} ^\ast]-\mathbb{J}[aX_{\lambda_1-\lambda_2}+yX_{\lambda_1+\lambda_2} ^\ast, X_{\lambda_1+\lambda_2}]\\
& = & a^2 m_{\lambda_1-\lambda_2 ,\lambda_1+\lambda_2}X_{2\lambda_1} - m_{\lambda_1-\lambda_2 ,\lambda_1+\lambda_2}X_{2\lambda_1}-am_{\lambda_1-\lambda_2,\lambda_1+\lambda_2}\mathbb{J}X_{2\lambda_1}\\
& &  - am_{\lambda_1-\lambda_2,\lambda_1+\lambda_2}\mathbb{J}X_{2\lambda_1}  \\
& = & (a^2 +1)m_{\lambda_1-\lambda_2,\lambda_1+\lambda_2}X_{2\lambda_1} - 2am_{\lambda_1-\lambda_2,\lambda_1+\lambda_2}\mathbb{J}X_{2\lambda_1}.
\end{eqnarray*}

Observe that if $a=0$, then $N_{\mathbb{J}}(X_{\lambda_1-\lambda_2},X_{\lambda_1+\lambda_2}) = m_{\lambda_1-\lambda_2,\lambda1+\lambda_2}X_{2\lambda_1} \neq 0$. Let us now suppose that $a\neq 0$ and write $\mathbb{J}X_{2\lambda_1} = \sum_{j=1} ^l c_jX_{2\lambda_k} + \sum_{j=1} ^l d_jX_{2\lambda_k} ^\ast$. If $N_\mathbb{J} \equiv 0$, then we obtain from the expression obtained above for $N_{\mathbb{J}}(X_{\lambda_1-\lambda_2},X_{\lambda_1+\lambda_2})$ that $c_j = 0$ for $j=2,3,\cdots,l$ and $d_j = 0$ for $j=1,2\cdots,l$. Therefore, we have that $\mathbb{J}X_{2\lambda_1} = c_1X_{2\lambda_1}$ which is a contradiction since $\mathbb{J}^2 = -1$ and $c_1 \in \mathbb{R}$.

So, we conclude:

\begin{proposition}
Let $\mathbb{F}$ be the maximal flag manifold of type $C_l$, with $l\geq 6$ and $l$ even. Then, the invariant generalized almost complex structures on $\mathbb{F}$ are not integrable.
\end{proposition}

Summing up, we have proved that:

\begin{theorem}\label{NoIntegrables}
No $GM_2$-maximal real flag manifold admits integrable $K$-invariant generalized almost complex structures.
\end{theorem}

\section{Generalized geometry on the special cases $B_2$, $G_2$, $A_3$, and $D_l$ with $l\geq 5$}\label{S:5}
Motivated by the study of several geometric aspects involving invariant generalized complex structures on complex flag manifolds developed in \cite{GVV}, in this section we give a concrete description of the generalized geometry on those maximal real flag manifolds of type $B_2$, $G_2$, $A_3$, and $D_l$ with $l\geq 5$. What makes these cases special among all maximal real flag manifolds is the fact that $\dim V_{[\alpha]_M}=2$ and because of this, for a fixed root $\alpha\in \Pi^{-}$, a generalized complex structure $\mathbb{J}_{[\alpha]_M}$ on $V_{[\alpha]_M}$ is given by either the structure $\mathcal{J}^c_{[\alpha]_M}$ or $\mathcal{J}^{nc}_{[\alpha]_M}$ introduced in Remark \ref{2DimensionalCases}. In the sequel $\mathbb{F}$ will denote any of the maximal real flag manifolds determined by the special cases mentioned above. Consider the set $\mathcal M_a(\mathbb{F})$ of all invariant generalized almost complex structures on $\mathbb F$. For a fixed root $\alpha \in \Pi^-$,
let $\mathcal M_\alpha(\mathbb{F})$ be the  restriction of $\mathcal M_a(\mathbb{F})$ to the subspace $V_{[\alpha]_M}$. Thus, the following result is clear:

\begin{lemma} \label{Aalpha}
	$\mathcal M_\alpha(\mathbb{F})$ is a disjoint union of:
	\begin{enumerate}
		\item[$\iota$.] a 2-dimensional family of structures of noncomplex type parametrized by the real algebraic surface
		cut out by $a_\alpha^2-x_\alpha y_\alpha =-1$ in $\mathbb R^3$, and 
		\item[$\iota\iota$.]  a 2-dimensional family of structures of complex type. 
	\end{enumerate}
\end{lemma}
It is worth mentioning that this clearly differs from what was known for maximal complex flags where there were only 2 isolated structures of complex type inside $\mathcal M_\alpha(\mathbb{F})$; compare \cite{GVV}.
\subsection{Effects of the action by invariant $B$-transformations} 
Let $\alpha \in \Pi^-$ be a fixed root. First of all, let us look at the effects of the action by $B$-transformations on the generalized complex structures on $V_{[\alpha]_M}$ which are of noncomplex type. Similar to how it was done in \cite{GVV}, if $\mathbb{J}_{[\alpha]_M}=\mathcal{J}^{nc}_{[\alpha]_M}$ is of noncomplex type with $a=0$ then because the equation $xy=1$ we get a structure of symplectic type since
\begin{equation}\label{ECType1}
\mathbb{J}_{[\alpha]_M}=\left( 
\begin{array}{cc}
0 & -\omega_\alpha^{-1}\\
\omega_\alpha & 0
\end{array}%
\right),
\end{equation}
where $\omega_\alpha=\dfrac{1}{x}\left(\begin{array}{cc}
0 & -1\\
1 & 0
\end{array}%
\right)\approx \dfrac{1}{x}X_\alpha^\ast\wedge X_\beta^\ast$, with $\alpha\sim_M \beta$, defines a symplectic form on $V_{[\alpha]_M}$. The structure showed in \eqref{ECType1} will be denoted by $\mathbb{J}_{\omega_\alpha}$. In general, if $\mathbb{J}_{[\alpha]_M}=\mathcal{J}^{nc}_{[\alpha]_M}$ is of noncomplex type verifying the equation $a^2=x y-1$ and $\mathbb{J}_{\omega_\alpha}$ is the structure of symplectic type defined as in \eqref{ECType1} using the nonzero parameter $x$, then it is simple to check that the equation $a^2=x y-1$ implies
\begin{equation}\label{AllSymplectic}
e^{-B_\alpha}\mathbb{J}_{\omega_\alpha}e^{B_\alpha}=\left( 
\begin{array}{cc}
-\omega_\alpha^{-1}B_\alpha & -\omega_\alpha^{-1}\\
\omega_\alpha+B_\alpha\omega_\alpha^{-1}B_\alpha& B_\alpha\omega_\alpha^{-1}
\end{array}%
\right)=\left( 
\begin{array}{cccc}
a & 0 & 0 & -x\\ 
0 &  a & x & 0\\
0 & -y & -a & 0\\
y & 0 & 0 & -a
\end{array}%
\right)=\mathcal{J}^{nc}_{[\alpha]_M},
\end{equation}
where $B_\alpha=\dfrac{a}{x} X_\alpha^\ast\wedge X_\beta^\ast$ with $\beta$ the other one root such that $\alpha\sim_M \beta$. In particular, $\textnormal{Type}(\mathbb{J}_{[\alpha]_M})=0$ and we have:
\begin{lemma}\label{LemmaBsymplectic}
Suppose that both $\mathbb{J}_{[\alpha]_M}$ and $\mathbb{J}_{[\alpha]_M}'$ are of noncomplex type with
$$\mathcal{J}^{nc}_{[\alpha]_M}={\tiny\left( 
	\begin{array}{cccc}
	a & 0 & 0 & -x\\ 
	0 &  a & x & 0\\
	0 & -y & -a & 0\\
	y & 0 & 0 & -a
	\end{array}%
	\right)} \quad \textnormal{and}\quad \mathcal{J'}^{nc}_{[\alpha]_M}={\tiny\left( 
	\begin{array}{cccc}
	a' & 0 & 0 & -x\\ 
	0 &  a' & x & 0\\
	0 & -y' & -a' & 0\\
	y' & 0 & 0 & -a'
	\end{array}%
	\right)}.$$
Then, there exists a $B$-transformation $B_\alpha\in \bigwedge^2 V_{[\alpha]_M}^\ast$ such that
$$e^{-B_\alpha}\mathbb{J}_\alpha e^{B_\alpha} =\mathbb{J}_\alpha'.$$
\end{lemma}
\begin{proof}
It is clear that $\mathbb{J}_{\omega_\alpha}=\mathbb{J}_{\omega_\alpha}'$. Thus, as consequence of the arguments given above there exist $\hat{B}_\alpha$ and $B_\alpha'$ in $\bigwedge^2 V_{[\alpha]_M}^\ast$ such that
$$e^{\hat{B}_\alpha}\mathbb{J}_\alpha e^{-\hat{B}_\alpha}=e^{B_\alpha'}\mathbb{J}_\alpha' e^{-B_\alpha'}.$$
Therefore, the identity $e^{\hat{B}+B'}=e^{\hat{B}}\cdot e^{B'}$ allows us to ensure that the result follows for $B_\alpha=B_\alpha'-\hat{B}_\alpha$.
\end{proof}

Let us now suppose that $\mathbb{J}_{[\alpha]_M}=\mathcal{J}^{c}_{[\alpha]_M}$ is of complex type. Here we have that $\mathcal{J}^{c}_{[\alpha]_M}=\left( 
\begin{array}{cc}
-J^c & 0\\
0 & (J^c)^\ast
\end{array}%
\right)$ where $J^c=\left( 
\begin{array}{cc}
-b & \dfrac{1+b^2}{c}\\
-c & b
\end{array}%
\right)$ with $c\neq 0$ is a complex structure on $V_{[\alpha]_M}$. In this case, a straightforward computation allows us to get that for every $B$-transformation $B\in \bigwedge^2 V_{[\alpha]_M}^\ast$ we obtain
$$e^{-B}\mathcal{J}^{c}_{[\alpha]_M}e^B=\left( 
\begin{array}{cc}
-J^c & 0\\
BJ^c+(J^c)^\ast B& (J^c)^\ast
\end{array}%
\right)=\mathcal{J}^{c}_{[\alpha]_M}.$$
This is because $BJ^c+(J^c)^\ast B=0$ for every $B=rX_\alpha^\ast\wedge X_\beta^\ast$ with $r\in\mathbb{R}$ and $\alpha \sim_M \beta$.
\begin{lemma}\label{LemmaBcomplex}
If $\mathbb{J}_{[\alpha]_M}$ is of complex type, then for every $B$-transformation $B\in \bigwedge^2 V_{[\alpha]_M}^\ast$ we have
$$e^{-B}\mathbb{J}_{[\alpha]_M} e^{B} =\mathbb{J}_{[\alpha]_M}.$$
\end{lemma}
In other words, the action by $B$-transformations leaves fixed every element in the 2-dimensional family of generalized complex structures of complex type on $V_{[\alpha]_M}$. Let  $\displaystyle \mathfrak{M}_\alpha(\mathbb{F})=\mathcal M_\alpha(\mathbb{F})/\mathcal{B}$ denote the moduli space of generalized complex structures on $V_{[\alpha]_M}$ under $B$-transformations. Thus, we obtain:
\begin{proposition}\label{Bmoduli1}
The quotient space $\displaystyle \mathfrak{M}_\alpha(\mathbb{F})$ consists of 2 disjoint sets:
\begin{enumerate}
\item[$\iota$.] a punctured real line $\mathbb{R}^\ast$ parametrizing structures of symplectic type, and 
\item[$\iota\iota$.] a plane minus a line $\mathbb{R}^\ast\times \mathbb{R}$ parametrizing structures of complex type.
\end{enumerate}
That is
$$\displaystyle \mathfrak{M}_\alpha(\mathbb{F})=\mathbb{R}^\ast \cup (\mathbb{R}^\ast\times \mathbb{R}).$$
In particular, $\mathfrak{M}_\alpha(\mathbb{F})$ admits a natural topology with which it is homotopy equivalent to $S^1$.
\end{proposition}
\begin{proof}
On the one hand, by Lemma \ref{LemmaBcomplex} we have that structures in $\mathcal M_\alpha(\mathbb{F})$ which are of complex type are fixed points of the action by $B$-transformations. This implies that up to action by $B$-transformations the whole 2-dimensional family of structures of complex type on $V_{[\alpha]_M}$ mentioned in Lemma \ref{Aalpha} is parametrized by the plane minus a line $\mathbb{R}^\ast\times \mathbb{R}$ holding the values of $(c,b)$. On the other hand, by Lemma \ref{Aalpha}, structures in $\mathcal M_\alpha(\mathbb{F})$ which are of noncomplex type are parameterized by a real surface $a^2 -x y=-1$ in $\mathbb{R}^3$, and from equation \eqref{AllSymplectic} we conclude that every point on this surface is the image by a $B$-transformation of a generalized complex structure of symplectic type determined by $x$. Hence the quotient of this real surface by $B$-transformations reduces to a punctured line $\mathbb{R}^\ast$ holding the values of $x$. If $x \neq x'$ are distinct then we get different generalized complex structures of symplectic type. Therefore, from Lemma \ref{LemmaBsymplectic} we conclude that points on this real punctured line represent non equivalent classes.

Finally, if we consider both sets $\mathbb{R}^\ast\approx \{0\}\times \mathbb{R}^\ast$ and $\mathbb{R}^\ast\times \mathbb{R}$ inside $\mathbb{R}^2$ with the subspace topology, then we may induce a natural topology to $\mathfrak{M}_\alpha(\mathbb{F})=\mathbb{R}^\ast \cup (\mathbb{R}^\ast\times \mathbb{R})$ which comes from the natural topology of $\mathbb{R}^2$. Thus, if we take the inclusions $\iota_1: \{0\}\times \mathbb{R}^\ast \hookrightarrow \mathbb{R}^2\backslash\{(0,0)\}$ and $\iota_2: \mathbb{R}^\ast\times \mathbb{R} \hookrightarrow \mathbb{R}^2\backslash\{(0,0)\}$, then using the pasting lemma we get a continuous function $f:\{0\}\times \mathbb{R}^\ast\cup(\mathbb{R}^\ast\times \mathbb{R})\to \mathbb{R}^2\backslash\{(0,0)\}$ which is actually a homeomorphism. In consequence, we may identify the quotient space $\mathfrak{M}_\alpha(\mathbb{F})\approx \mathbb{R}^2\backslash\{(0,0)\}$ and hence $\mathfrak{M}_\alpha(\mathbb{F})$ is homotopy equivalent to $S^1$.
\end{proof}

The proof of the next result follows the same steps used to prove \cite[Theorem 5.13]{GVV}. For this reason we will only give a few details about it. This result is essential for two reasons. The first one is because this allows us to describe the moduli space $\mathfrak{M}_a(\mathbb{F})$ for the maximal real flag manifolds $\mathbb{F}$ that we are dealing with in this section. The second one is because as consequence of this result and Lemmas \ref{LemmaBsymplectic} and \ref{LemmaBcomplex} we may give an explicit expression for the invariant pure spinor associated to each element in $\mathcal{M}_a(\mathbb{F})$.
\begin{proposition}\label{BModuli2}
Let $\mathbb{J}$ and $\mathbb{J}'$ be two invariant generalized almost complex structures on $\mathbb{F}$ such that for each $M$-equivalence class $[\alpha]_M$ the following conditions hold true:
\begin{enumerate}
\item[$\iota$.] if $\mathbb{J}_{[\alpha]_M}$ is of complex type, then $\mathbb{J}_{[\alpha]_M}=\mathbb{J}_{[\alpha]_M}'$, and
\item[$\iota\iota$.] if $\mathbb{J}_{[\alpha]_M}$ is of noncomplex type, then $\mathbb{J}_{[\alpha]_M}'$ is also of noncomplex type with
$$\mathcal{J}_{[\alpha]_M}^{nc}=\left( 
\begin{array}{cc}
\mathcal{A}_{\alpha} & \mathcal{X}_{\alpha} \\
\mathcal{Y}_{\alpha}  & -\mathcal{A}_{\alpha}
\end{array}%
\right) \quad \textnormal{and}\quad \mathcal{J'}_{[\alpha]_M}^{cn}=\left( 
\begin{array}{cc}
\mathcal{A}_{\alpha}' & \mathcal{X}_{\alpha} \\
\mathcal{Y}_{\alpha}'  & -\mathcal{A}_{\alpha}'
\end{array}%
\right).$$
\end{enumerate}
Then there exists a $B$-transformation $B\in \wedge^2 (\mathfrak{n}^-)^\ast$ such that
$$e^{-B}\mathbb{J}e^{B}=\mathbb{J}'.$$
\end{proposition}
\begin{proof}
Suppose that $\displaystyle \Pi^-=\bigcup_{j=1}^d[\alpha_j]_M$ where $d$ is the number of $M$-equivalence classes. For each $M$-equivalence class $[\alpha_j]_M$ we set 
$$\mathbb{J}_{[\alpha_j]_M}=\left( 
\begin{array}{cc}
\mathcal{A}_{\alpha_j} & \mathcal{X}_{\alpha_j} \\
\mathcal{Y}_{\alpha_j}  & -\mathcal{A}_{\alpha_j}
\end{array}%
\right)  \quad \textnormal{and}\quad  \mathbb{J}_{[\alpha_j]_M}'=\left( 
\begin{array}{cc}
\mathcal{A}_{\alpha_j}' & \mathcal{X}_{\alpha_j} \\
\mathcal{Y}_{\alpha_j}'  & -\mathcal{A}_{\alpha_j}'
\end{array}%
\right).$$
where by hypotheses if $\mathbb{J}_{[\alpha_j]_M}$ is of complex type, then $\mathbb{J}_{[\alpha_j]_M}=\mathbb{J}_{[\alpha_j]_M}'=\mathcal{J}_{[\alpha_j]_M}^{c}$ which implies that $\mathcal{A}_{\alpha_j}=\mathcal{A}_{\alpha_j}'=-J^{c}$ and $-\mathcal{A}_{\alpha_j}=-\mathcal{A}_{\alpha_j}'=(J^{c})^\ast$ and $\mathcal{X}_{\alpha_j}=\mathcal{Y}_{\alpha_j}'=\mathcal{Y}_{\alpha_j}=0$. Otherwise, if $\mathbb{J}_{[\alpha_j]_M}$ is of noncomplex type, then $\mathbb{J}_{[\alpha_j]_M}'$ is also of noncomplex type.

By Lemmas \ref{LemmaBsymplectic} and \ref{LemmaBcomplex} we have that for all $j=1,2,\cdots,d$ there exist $B$-transformations $B_j\in\bigwedge^2V_{[\alpha_j]_M}^\ast$ such that
	\begin{equation}\label{tired}
	e^{-B_j}\mathbb{J}_{[\alpha_j]_M} e^{B_j} =\mathbb{J}_{[\alpha_j]_M}'.
	\end{equation}
	If we define $B\in\bigwedge^2(\mathfrak{n}^-)^\ast$ by
	$B={\tiny \left( 
		\begin{array}{cccc}
		B_{1} &  &  & \\ 
		&  B_{2}&  & \\
		&  & \ddots & \\
		&  &  & B_{d}
		\end{array}%
		\right)}$, then setting
	$$\mathbb{J}={\tiny\left( 
		\begin{array}{ccccccccc}
		\mathcal{A}_{\alpha_1} &  &  &  & \mathcal{X}_{\alpha_1} & & \\
		& \mathcal{A}_{\alpha_2} &  &  & &  \mathcal{X}_{\alpha_2}& \\
		
		&  & \ddots & & &   &\ddots \\
		&  &  & \mathcal{A}_{\alpha_d} & &  & & \mathcal{X}_{\alpha_d} \\
		\mathcal{Y}_{\alpha_1}   & &  &  & -\mathcal{A}_{\alpha_1} & &\\
		& \mathcal{Y}_{\alpha_2} &  &  &  & -\mathcal{A}_{\alpha_2} &\\
		&  & \ddots & & &   &\ddots\\
		&  &  &\mathcal{Y}_{\alpha_d} & &   & & -\mathcal{A}_{\alpha_d}\\
		
	\end{array}%
	\right)},$$
we get that the identities deduced from \eqref{tired} imply that $e^{-B}\mathbb{J}e^{B}=\mathbb{J}'$.
\end{proof}
As an immediate consequence of Propositions \ref{Bmoduli1} and \ref{BModuli2} we obtain:
\begin{corollary}\label{BModuli2.1}
Suppose that $\displaystyle \Pi^-=\bigcup_{j=1}^d[\alpha_j]_M$ where $d$ is the number of $M$-equivalence classes. Then
$$\mathfrak M_a (\mathbb F)= \prod_{[\alpha_j]_M\subset\Pi^-} \mathfrak{M}_{\alpha_j}(\mathbb{F})=(\mathbb{R}^\ast \cup (\mathbb{R}^\ast\times \mathbb{R}))_{\alpha_1} \times \cdots \times (\mathbb{R}^\ast \cup (\mathbb{R}^\ast\times \mathbb{R}))_{\alpha_d}.$$
In particular, $\mathfrak M_a (\mathbb F)$ admits a natural topology induced from $\mathbb{R}^{2d}$ with which it is homotopy equivalent to the $d$-torus $\mathbb{T}^d$.
\end{corollary}
Let us decompose the set of roots $\Pi^-= \Pi^-_{nc}\cup\Pi^-_{c}$ where $\mathbb{J}_{[\alpha]_M}=\mathcal{J}_{[\alpha]_M}^{cn}$ is of noncomplex type for all $[\alpha]_M\subset \Pi^-_{nc}$ and $\mathbb{J}_{[\alpha]_M}=\mathcal{J}_{[\alpha]_M}^{c}$ is of complex type for all $[\alpha]_M\subset \Pi^-_{c}$. It is clear that because of the invariance all our generalized almost complex structures $\mathbb{J}$ are regular. In particular, $\textnormal{Type}(\mathbb{J})=\textnormal{Type}(\mathbb{J})_{b_0}=\vert \Pi^-_{c}/\sim_M\vert$. From previous results, when $\Pi^-_{c}=\emptyset$ we get that $\mathbb{J}$ is a $B$-transformation of a structure of symplectic type and when $\Pi^-_{c}=\Pi^-$ we obtain that $\mathbb{J}$ is of complex type. Otherwise, when $\Pi^-_{nc}\subset\Pi^-$ is nonempty we get a generalized almost complex structure which is neither complex nor symplectic.

Since finite products of generalized almost complex manifolds are still generalized complex with the obvious induced structure \cite{H,G3}, as an immediate consequence of Lemmas \ref{LemmaBsymplectic}, \ref{LemmaBcomplex} and Proposition \ref{BModuli2} we get an explicit expression for the invariant pure spinor associated to each structure in $\mathcal{M}_a(\mathbb{F})$.

\begin{corollary}\label{PureSpinor}
Let $\mathbb{J}$ be an invariant generalized almost complex structure on $\mathbb{F}$. Then the invariant pure spinor line $K_\mathcal{L}<\bigwedge^\bullet (\mathfrak{n}^-)^\ast\otimes \mathbb{C}$ associated to $\mathbb{J}$ is generated by
$$\varphi=e^{\sum_{[\alpha]_M\subset \Pi^-_{nc}}(B_\alpha +i\omega_\alpha)}\bigwedge_{[\alpha]_M\subset \Pi^-_{c}}\Omega_\alpha,$$
where $B_\alpha=\dfrac{a}{x}X_\alpha^\ast\wedge X_\beta^\ast$ and $\omega_\alpha=\dfrac{1}{x}X_\alpha^\ast\wedge X_\beta^\ast$ with $\alpha\sim_M \beta$ for all $[\alpha]_M\subset \Pi^-_{nc}$ and $\Omega_\alpha\in \wedge^{1,0}V_{[\alpha]_M}^\ast$ defines a complex structure on $V_{[\alpha]_M}$ for all $[\alpha]_M\subset \Pi^-_{c}$.
\end{corollary}

As $\varphi$ is initially defined on $T_{b_0}\mathbb{F}=\mathfrak{n}^-$ we may also use  invariance to define $\varphi$ on $\mathbb{F}$ as follows. Assume that $\varphi\in \bigwedge^r (\mathfrak{n}^-)^\ast\otimes \mathbb{C}$. Thus, at $x=gM\in \mathbb{F}=K/M$ for some $g\in K$ we define
$$\widetilde{\varphi}_x(X_1(x),\cdots, X_r(x))=\varphi(\Ad(g^{-1})_{\ast,x}X_1(x),\cdots,\Ad(g^{-1})_{\ast,x}X_r(x)).$$
In this case $\widetilde{\varphi}\in \bigwedge^r T^\ast\mathbb{F}\otimes \mathbb{C}$ would be a pure spinor determined by $\mathbb{J}: \mathbb{TF}\to \mathbb{TF}$ which satisfies $(\Ad(g))^\ast\widetilde{\varphi}=\widetilde{\varphi}$ for all $g\in K$. The same thing that we did with $\varphi$ can be done with the $B$-transformation constructed in Proposition \ref{BModuli2}.

\subsection{Invariant generalized almost Hermitian structures}
Let us now classify the invariant generalized almost Hermitian structures on $\mathbb{F}$. Given the requirement of invariance we need to find pairs of commuting invariant  generalized almost complex structures $(\mathbb{J},\mathbb{J'})$ such that $G:=-\mathbb{J}\mathbb{J'}$ defines a positive definite metric on $\mathfrak{n}^-\oplus (\mathfrak{n}^-)^\ast$. Because every invariant generalized almost complex structure on $\mathbb{F}$ has the form  $\displaystyle \mathbb{J} = \sum_{[\alpha]_M} \mathbb{J}_{[\alpha]_M}$, we just need to determinate those commuting pairs $(\mathbb{J}_{[\alpha]_M},\mathbb{J'}_{[\alpha]_M})$ such that $G_\alpha=-\mathbb{J}_{[\alpha]_M}\mathbb{J'}_{[\alpha]_M}$ defines a positive definite metric on $V_{[\alpha]_M}\oplus V_{[\alpha]_M}^\ast$ of signature $(2,2)$ for all $[\alpha]_M\subset \Pi^-$. Thus, latter requirement allows us to conclude that we only have to work with pairs of the form $(\mathcal{J}^c_{[\alpha]_M},\mathcal{J}^{nc}_{[\alpha]_M})$.

Let $\alpha\in\Pi^-$ be a fixed root. After a straightforward computation we may easily see that $\mathcal{J}^c_{[\alpha]_M}$ and $\mathcal{J}^{nc}_{[\alpha]_M}$ always commute. Indeed,
\begin{equation}\label{GKhaler1}
\mathcal{J}^c_{[\alpha]_M}\mathcal{J}^{nc}_{[\alpha]_M}={\tiny\left( 
\begin{array}{cccc}
ab & -\dfrac{a(1+b^2)}{c} & -\dfrac{x(1+b^2)}{c} & -bx\\ 
ac &  -ab & -bx & -cx\\
-cy & by & ab & ac\\
by & -\dfrac{y(1+b^2)}{c} & -\dfrac{a(1+b^2)}{c} & -ab
\end{array}%
\right)}=\mathcal{J}^{nc}_{[\alpha]_M}\mathcal{J}^c_{[\alpha]_M}.
\end{equation}
We will mainly use Proposition \ref{ModuliM}, Lemma \ref{LemmaBsymplectic}, and the fact concluded from Equation \eqref{AllSymplectic}. Let $\mathbb{J}_{\omega_\alpha}$ be the structure of symplectic type defined by means of $\mathcal{J}^{nc}_{[\alpha]_M}$ using the parameter $x$ (look at Expression \eqref{ECType1}) and let us consider
\begin{equation}\label{GKhaler1.1}
\widetilde{G}_\alpha:=-\mathbb{J}_{\omega_\alpha}\mathcal{J}^c_{[\alpha]_M}=-\mathcal{J}^c_{[\alpha]_M}\mathbb{J}_{\omega_\alpha}={\tiny\left( 
\begin{array}{cccc}
0 & 0 & \dfrac{x(1+b^2)}{c} & bx\\ 
0 &  0 & bx & cx\\
\dfrac{c}{x} & -\dfrac{b}{x} & 0 & 0\\
-\dfrac{b}{x} & \dfrac{1+b^2}{cx} & 0 & 0
\end{array}%
\right)}.
\end{equation}
It is simple to check that $\widetilde{G}_\alpha=\left( 
\begin{array}{cc}
0 & g^{-1}\\
g & 0
\end{array}%
\right)$ where $g=\left( 
\begin{array}{cc}
	\dfrac{c}{x} & -\dfrac{b}{x}\\
	-\dfrac{b}{x} & \dfrac{1+b^2}{cx}
\end{array}%
\right)$.
\begin{lemma}\label{GKhaler2}
The matrix $g$ allows us to define an inner product on $V_{[\alpha]_M}$ if and only if $cx>0$.
\end{lemma}
\begin{proof}
Since $g$ is clearly symmetric, we only need to look at the conditions under which the $2$ eigenvalues of $g$ are positive. They are given by
$$\lambda_1=\dfrac{c^2+b^2+1+\sqrt{(c^2+b^2+1)^2-4c^2}}{2cx}\quad \textnormal{and} \quad \lambda_2=\dfrac{c^2+b^2+1-\sqrt{(c^2+b^2+1)^2-4c^2}}{2cx}.$$
Recall that both $x$ and $c$ are nonzero. Obviously $\lambda_1>0$ if and only if $cx>0$ and because $c^2+b^2+1\geq \sqrt{(c^2+b^2+1)^2-4c^2}$ always holds true, this is the same condition for $\lambda_2$ being positive.
\end{proof}
Given the generalized complex structures $\mathcal{J}^c_{[\alpha]_M}$ and $\mathcal{J}^{nc}_{[\alpha]_M}$, we denote them respectively by $(\mathcal{J}^c_{[\alpha]_M})^+$ and $(\mathcal{J}^{nc}_{[\alpha]_M})^+$ if $c>0$ and $x>0$. Otherwise, we denote them respectively by $(\mathcal{J}^c_{[\alpha]_M})^-$ and $(\mathcal{J}^{nc}_{[\alpha]_M})^-$. So, in terms of this notation we have:
\begin{proposition}\label{GKhaler3}
	Let $(\mathbb{J},\mathbb{J}')$ be an invariant generalized almost Hermitian structure on $\mathbb{F}$. Then for every $M$-equivalence class $[\alpha]_M\subset \Pi^-$, 
	the pair  $(\mathbb{J}_{[\alpha]_M},\mathbb{J}_{[\alpha]_M}')$ takes one of the following values
	\begin{center}
		\begin{tabular}{c|c}
			$\mathbb{J}_{[\alpha]_M}$ & $\mathbb{J}_{[\alpha]_M}'$\\
			\hline
			$(\mathcal{J}^c_{[\alpha]_M})^+$ & $(\mathcal{J'}^{cn}_{[\alpha]_M})^+$ \\
			$(\mathcal{J}^{cn}_{[\alpha]_M})^+$ & $(\mathcal{J'}^c_{[\alpha]_M})^+$ \\
			$(\mathcal{J}^c_{[\alpha]_M})^-$ & $(\mathcal{J'}^{nc}_{[\alpha]_M})^-$  \\
			$(\mathcal{J}^{nc}_{[\alpha]_M})^-$ & $(\mathcal{J'}^c_{[\alpha]_M})^-$.
		\end{tabular}
	\end{center}
\end{proposition}
\begin{proof}
As consequence of Lemma \ref{GKhaler2} the values of the previous table are precisely those that either the pair $(\mathbb{J}_{\omega_\alpha},\mathcal{J'}^c_{[\alpha]_M})$ or $(\mathcal{J}^c_{[\alpha]_M},\mathbb{J'}_{\omega'_\alpha})$ can take. In any case we get a generalized metric $\widetilde{G}_\alpha$ as that given by Equation \eqref{GKhaler1.1}. By Lemma \ref{LemmaBsymplectic} and Equation \eqref{AllSymplectic} we have that
$$e^{B_\alpha}\cdot (\mathbb{J}_{\omega_\alpha},\mathcal{J'}^c_{[\alpha]_M})=(e^{-B_\alpha}\mathbb{J}_{\omega_\alpha}e^{B_\alpha},e^{-B_\alpha}\mathcal{J'}^c_{[\alpha]_M}e^{B_\alpha})=(\mathbb{J}_{[\alpha]_M},\mathcal{J'}^c_{[\alpha]_M})=(\mathbb{J}_{[\alpha]_M},\mathbb{J'}_{[\alpha]_M}),$$
for $B_\alpha=\dfrac{a}{x}X_\alpha^\ast\wedge X_\beta^\ast$ with $\alpha\sim_M \beta$. Analogously, $e^{B'_\alpha}\cdot(\mathcal{J}^c_{[\alpha]_M},\mathbb{J'}_{\omega'_\alpha})=(\mathbb{J}_{[\alpha]_M},\mathbb{J'}_{[\alpha]_M})$ for $B'_\alpha=\dfrac{a'}{x'}X_\alpha^\ast\wedge X_\beta^\ast$. Therefore, by Proposition \ref{ModuliM}, the generalized metric associated to the pair $(\mathbb{J}_{[\alpha]_M},\mathbb{J'}_{[\alpha]_M})$ is either $G_\alpha=e^{-B_\alpha}\widetilde{G}_\alpha e^{B_\alpha}$ or $G_\alpha'=e^{-B'_\alpha}\widetilde{G'}_\alpha e^{B'_\alpha}$ and, in any case, we precisely obtain a matrix of the form minus the matrix given in Equation \eqref{GKhaler1}. Last claim is consequence of having the constrains $a^2=xy-1$ or $a'^2=x'y'-1$ for the structures of noncomplex type. So, the result follows.
\end{proof}
Let  $\displaystyle \mathfrak{K}_\alpha(\mathbb{F})=\mathcal K_\alpha(\mathbb{F})/\mathcal{B}\subset \mathfrak{M}_\alpha(\mathbb{F})\times \mathfrak{M}_\alpha(\mathbb{F})$ denote the moduli space of generalized almost Hermitian structures on $V_{[\alpha]_M}$ under $B$-transformations. Thus, we obtain
\begin{corollary}
Suppose that $\displaystyle \Pi^-=\bigcup_{j=1}^d[\alpha_j]_M$ where $d$ is the number of $M$-equivalence classes. Then
$$\mathfrak K_a (\mathbb F)= \prod_{[\alpha_j]_M\subset\Pi^-} \mathfrak{K}_{\alpha_j}(\mathbb{F})=\mathbb{R}^\dagger_{\alpha_1} \times \cdots \times \mathbb{R}^\dagger_{\alpha_d},$$
where $\mathbb{R}^\dagger=\lbrace\mathbb{R}^+\times(\mathbb{R}^+\cup \mathbb{R}) \rbrace\cup \lbrace (\mathbb{R}^+\cup \mathbb{R})\times \mathbb{R}^+ \rbrace\cup\lbrace\mathbb{R}^-\times(\mathbb{R}^-\cup \mathbb{R}) \rbrace\cup \lbrace (\mathbb{R}^-\cup \mathbb{R})\times \mathbb{R}^- \rbrace$. Moreover, $\mathfrak K_a (\mathbb F)$ admits a natural topology induced by the product topology of $\mathfrak M_a (\mathbb F)\times \mathfrak M_a (\mathbb F)$.
\end{corollary}
\begin{proof}
As consequence of Proposition \ref{GKhaler3} and the arguments used in its proof, we easily conclude that $\mathfrak{K}_{\alpha_j}(\mathbb{F})=\mathbb{R}^\dagger$ for every $M$-equivalence class $[\alpha_j]_M\subset\Pi^-$. Therefore, the result immediately follows from Proposition \ref{BModuli2} and Corollary \ref{BModuli2.1}.
\end{proof}
Finally, let $\mathcal{G}_a(\mathbb{F})$ denote the set of all invariant generalized metrics on $\mathbb{F}$. Motivated by Proposition \ref{ModuliM} we can define an action of $\mathcal{B}$ on $\mathcal{G}_a(\mathbb{F})$ as $e^B\cdot G=e^{-B}Ge^B$. The quotient space $\mathfrak{G}_a(\mathbb{F}):=\mathcal{G}_a(\mathbb{F})/\mathcal{B}$ induced by this action is called {\it moduli space of invariant generalized metrics on $\mathbb{F}$ under invariant $B$-transformations}; see \cite{GVV}. 
\begin{corollary}\label{GMetrics}
Suppose that $\displaystyle \Pi^-=\bigcup_{j=1}^d[\alpha_j]_M$ where $d$ is the number of $M$-equivalence classes. Then
$$\mathfrak{G}_a(\mathbb{F})=\lbrace((\mathbb{R}^+)^2\times\mathbb{R}) \cup ((\mathbb{R}^-)^2\times\mathbb{R})\rbrace_{\alpha_1}\times \cdots \times \lbrace((\mathbb{R}^+)^2\times\mathbb{R}) \cup ((\mathbb{R}^-)^2\times\mathbb{R})\rbrace_{\alpha_d}.$$
\end{corollary}
\begin{proof}
As we saw before, up to action by $B$-transformations a generalized metric on $V_{[\alpha]_M}$ has the form $\widetilde{G}_\alpha=\left( 
\begin{array}{cc}
0 & g^{-1}\\
g & 0
\end{array}%
\right)$ where
$$g=\left( 
\begin{array}{cc}
\dfrac{c}{x} & -\dfrac{b}{x}\\
-\dfrac{b}{x} & \dfrac{1+b^2}{cx}
\end{array}%
\right)\qquad \textnormal{verifying}\qquad cx>0.$$
Therefore, up to action by $B$-transformations we get that the set of generalized metrics on $V_{[\alpha]_M}$ is parametrized by $((\mathbb{R}^+)^2\times\mathbb{R}) \cup ((\mathbb{R}^-)^2\times\mathbb{R})$. So, the result follows from Proposition \ref{BModuli2} as desired.
\end{proof}

\end{document}